\documentclass{eectOF}
\usepackage{amsmath}
\usepackage{amssymb}
  \usepackage{paralist}
  \usepackage{graphics} 
  \usepackage{epsfig} 
\usepackage{graphicx}  \usepackage{epstopdf}
 \usepackage[colorlinks=true]{hyperref}
\hypersetup{urlcolor=blue, citecolor=red}

\def\currentvolume{X}
 \def\currentissue{X}
  \def\currentyear{200X}
   \def\currentmonth{XX}
    \def\ppages{X--XX}
     \def\DOI{10.3934/eect.2021044}

\usepackage{bigints}
\usepackage{empheq}
\allowdisplaybreaks
\usepackage{dsfont}
\usepackage{amssymb}
\usepackage{amsfonts}
\usepackage{amsmath}
\usepackage{graphicx}
\usepackage{shadow}
\usepackage{color}
\usepackage{tcolorbox}
\usepackage[all]{xy}
\usepackage{tikz}
\usepackage{graphicx}
\usepackage{enumitem}
\usepackage{cases}

\newtheorem{theorem}{Theorem}[section]
\newtheorem{corollary}{Corollary}

\newtheorem{lemma}[theorem]{Lemma}
\newtheorem{proposition}{Proposition}

\theoremstyle{definition}
\newtheorem{definition}[theorem]{Definition}
\newtheorem{remark}{Remark}

\def \dv {\mathrm{div}}

\title[Stackelberg-Nash null controllability] 
      {Stackelberg-Nash null controllability of heat equation with general  dynamic\\ boundary conditions}

\author[Idriss Boutaayamou,  Lahcen Maniar and Omar Oukdach]{}


 \email{dsboutaayamou@gmail.com}
 \email{maniar@uca.ma}
 \email{omar.oukdach@gmail.com}

\thanks{$^*$ Corresponding author: Idriss Boutaayamou}

\begin{document}
\maketitle

\centerline{\scshape Idriss Boutaayamou$^*$}

\medskip
{\footnotesize
	
	\centerline{  Facult\'e Polydisciplinaire de Ouarzazate,}
	\centerline{ Universit\'e Ibn Zohr, Ouarzazate,  B.P. 638, 45000, Morocco }
} 

\bigskip
\centerline{\scshape Lahcen Maniar and Omar Oukdach}

\medskip
{\footnotesize
   \centerline{Cadi Ayyad University, Faculty of Sciences Semlalia}
   \centerline{	LMDP, UMMISCO (IRD-UPMC) B.P. 2390, Marrakesh, Morocco}
} 

\bigskip

 \centerline{(Communicated by Genni Fragnelli)}

\begin{abstract}
	This paper deals with the hierarchical control of the anisotropic heat equation  with dynamic boundary conditions and drift terms. We use the Stackelberg-Nash strategy with one leader and two followers. To each fixed leader, we find  a Nash equilibrium corresponding to a bi-objective optimal control problem for the followers. Then, by some new Carleman estimates, we prove a null controllability result.
\end{abstract}
\section{Introduction}

A control  problem consists of finding a control that steers the system under consideration from an initial state to a fixed target. In many situations, the resolution process leads to a single-objective optimal control problem, and  generally, under some appropriate conditions, the existence and uniqueness of  optimal control problems can be proved.
\par The modeling of  many  industrial and economic complex problems reveals   several criteria to optimize simultaneously. For example,
in the heat transfer, in a room, it is meaningful to try to guide the temperature
to be close as much as possible to a fixed target  at the end of the day and, additionally,
keep the temperature not too far from a prescribed value at some
regions. This can be done by applying several controls at different locations of the
room. This problem can be seen as a game with controls as players.

\par To solve a  multi-objective optimal control problem, there are several strategies  to choose the best controls, depending on the character of the problem. Among these strategies, let us mention the cooperative strategy proposed by  Pareto  \cite{Na51}, the Nash non-cooperative strategy  \cite{Pa96},  and the Stackelberg hierarchical strategy \cite{St34}.

In the context of the control of evolution equations, a relevant question is whether
one can  steer the system to the desired state exactly or approximately by
controls  corresponding to one of these strategies. Up to now, there are
several papers related to this topic. In the seminal papers \cite{LiHy} and \cite{LiPa}, J. L. Lions  introduced and studied  the Stackelberg strategy  for the hyperbolic and parabolic equations.  Later on, the authors in  \cite{D02} and \cite{DL04},  studied the existence and uniqueness of Stackelberg-Nash equilibrium, as well as its characterization. In these works, the followers and the leader have  an approximate controllability objective. Then again, from a theoretical and numerical point of view, the papers  \cite{CF18, FACC18, GRP02, GRP01} dealt with the above questions, for parabolic equations and the Burgers equation, respectively.  We refer also to \cite{LCM09}, where the authors obtained  results of hierarchical control for parabolic equations with moving boundaries. Eventually, let  us mention that the Stackelberg-Nash strategy for  Stokes systems has been studied
in \cite{GMR13}.
We emphasize that   all the above results  deal with the hierarchical control for the evolution equation only in the case of approximate controllability. In \cite{ArCaSa15}, the authors developed the first hierarchical results in the context of  the  controllability to trajectories. These results were recently improved in \cite{ArFeGu17} by imposing some weak conditions on observation domains for the followers. The same idea is also applied for the wave equation and  the degenerate parabolic equations in the papers \cite{FACC18}  and \cite{AAF}, respectively. More recently, in \cite{AFS},  the authors dealt with the  hierarchical    exact controllability of  parabolic
equations with distributed and boundary controls. The same method was also applied in the context of parabolic coupled systems in  \cite{HTP16} and \cite{HSP18} and mixed with robust control in \cite{Luz18} and \cite{Mophou1}.  The previous works have considered Dirichlet and Neumann boundary conditions.

In this paper, we are interested in developing the Stackelberg-Nash strategy for a parabolic equation with dynamic boundary conditions.
This type of boundary conditions has been considered in the context of null controllability, when only one control is acting on the system, in  \cite{KhMa19, MMS17}. The cost of approximate controllability and an inverse source problem of such equations  were also studied in \cite{BCMO20} and \cite{ACMO20}, respectively. For some controllability results of hyperbolic equation with dynamic boundary condition, the reader can see \cite{CiGa'17}. We refer  to \cite{Gol06} for a derivation and physical interpretation of such boundary conditions, and to \cite{CM09,  GMS10, KhMaMaGh19, MiZi05, B10, C2, B11} for  the study of the  existence, the uniqueness and the regularity of their solutions. Note also
that the  controllability problems of evolution equations is extensively studied in the case of static boundary conditions in  \cite{AFHM11,BGC06,BFM18,CF07, PWZ07} and reference therein. It is worth mentioning that the presence of dynamic boundary conditions, in evolution problems, creates a transmission problem between the dynamics in the domain and on the boundary, also appear in various interesting physical models and motivated by problems occurring diffusion, reaction-diffusion, and phase-transition phenomena, special flows in hydrodynamics, models in climatology, and so on. We refer to \cite{Gol06} for a derivation and physical interpretation of such boundary conditions.

Our main results rely on a new Carleman estimate of a coupled system. In our case, the first equation is a direct heat equation those the second term is in $\mathbb{L}^{2}$ and an adjoint equation with  terms in a Sobolev space of negative order $\mathbb{H}^{-1}$.  To our knowledge, our paper is the first to deal with Carleman estimates for a coupled system with  terms in $\mathbb{H}^{-1}$, especially in the context of hierarchical control of evolution equations.

This paper is organized as follows. In Section \ref{S.1},  we   formulate the problem under consideration and state the main result. In Section \ref{S.2}, we introduce  some needed functional spaces and recall some previous results on the well-posedness. Section \ref{S.3} deals with the proof of existence and characterization of Nash equilibrium. In  Section \ref{S.4}, we prove some suitable Carleman estimates and deduce our null controllability result in the linear case. The semilinear case is considered in Section \ref{sectionsemli}.


\section{Problem and its formulation}\label{S.1}
Let  $\Omega$ be a bounded open set of $\mathbb{R}^N$, $N \geq 2$, with smooth boundary $\Gamma :=\partial\Omega$. Let    $\omega$, $\omega_1$ and $\omega_2$ be  arbitrary nonempty open subsets  strictly contained in $\Omega$. Given $T>0$, we denote   $\Omega_T:=\Omega\times(0, T)$ and  $\Gamma_T:=\Gamma\times(0, T)$.
Consider the heat equation with dynamic boundary conditions
\begin{equation}\label{Sys1.1SN}
\left\{\begin{array}{ll}
{\partial_t y-\mathrm{div}(\mathcal{A}\nabla y)  +B(x,t)\cdot \nabla y+ a(x,t)y= f1_{\omega}+ v_1 1_{\omega_1}+ v_2 1_{\omega_2}}& {\text { in } \Omega_{T},} \\
{\partial_t y_{_{\Gamma}}   - \dv_{\Gamma}(\mathcal{A}_{\Gamma}\nabla_{\Gamma} y_{\Gamma})    + \partial_{\nu}^{\mathcal{A}}y+B_{\Gamma}(x, t)\cdot \nabla_{\Gamma} y_{\Gamma} + b(x,t)y_{\Gamma}=0}& {\text { on } \Gamma_{T},} \\
{(y(0),y_{\Gamma}(0)) =(y_0, y_{\Gamma,0})},& {\text { in } \Omega \times \Gamma}.
\end{array}\right.
\end{equation}	
Here, $y_{\Gamma}$ denotes the trace of $y$ on $\Gamma$, $y_0 \in L^2(\Omega)$ and $y_{\Gamma, 0} \in  L^2(\Gamma)$ the initial data
in the domain and  on the
boundary, respectively. We emphasize that $y_{\Gamma,0}$
is not necessarily the trace of $y_{0}$, since we do not assume that $y_0$ has a trace, but if $y_0$ has a
well-defined trace on $\Gamma$, then the trace must coincide with  $y_{\Gamma, 0}$, and
$\nu$ is the outer unit normal field, $  \partial^{\mathcal{A}}_{\nu}y:= (\mathcal{A}\nabla y\cdot\nu)_{|_{\Gamma}}$ is the
co-normal derivative at $\Gamma$.
The functions   $a$, $b$, $B$ and $B_\Gamma$  belong to   $L^{\infty}(\Omega_T)$, $L^{\infty}(\Gamma_T)$, $L^{\infty}(\Omega_T)^N$ and  $L^{\infty}(\Gamma_T)^N$, respectively;  $1_\omega$, $1_{\omega_1}$ and $1_{\omega_2}$  are the  characteristic functions of $\omega$, $\omega_1$ and $\omega_2$, respectively, and $f$, $v_1$ and $v_2$ are the control functions which act on the system through the subsets $\omega$, $\omega_1$ and $\omega_2$, respectively. The boundary $\Gamma$ of $\Omega$ is considered to be a $(N -1)$-dimensional compact Riemannian submanifold equipped by  the Riemannian metric $g$ induced by the natural embedding $\Gamma \subset \mathbb{R}^N$. Let $(x_i)$ be the natural coordinate system, we denote ($\frac{\partial}{\partial x_i}$) the corresponding tangent vector field and $g_{i j}=\langle \frac{\partial}{\partial x_i}, \frac{\partial}{\partial x_j}\rangle$. We define the inner product and the norm on the tangent space by
$$g(X, Y)=\langle X, Y\rangle_{\Gamma}=\sum\limits_{i j=1}^{N-1} g_{i j}(x) \alpha_{i} \beta_{j}\,\, \text{and}\,\,
|X|_{g}=\langle X, X\rangle_{\Gamma}^{\frac{1}{2}}
$$
for all  $\displaystyle X=\sum\limits_{i=1}^{N-1} \alpha_{i} \frac{\partial}{\partial x_{i}}$ and  $\displaystyle Y=\sum\limits_{i=1}^{N-1} \beta_{i} \frac{\partial}{\partial x_{i}}$ in the tangent space and we simply denote $X\cdot Y= \langle X, Y\rangle_{\Gamma}$. For any smooth function $y$ on $\Gamma$, the tangential gradient  of $y$ on $\Gamma$ is defined as
$\displaystyle\nabla_{\Gamma} y=\sum_{i, j=1}^{N-1} g^{i j} \frac{\partial y}{\partial x_{j}} \frac{\partial}{\partial x_{i}},$
where the matrix $\hat{g}=(g^{i j})$ is the inverse of $g=(g_{i j})$, see \cite{ LaTrYa99} and  \cite{C2}   for more detail. We also recall, from \cite{KhMa19}, the following weak definitions of the  divergence operator $\dv(\cdot)$ and the  tangential divergence operator $\dv_{\Gamma}(\cdot)$.
For $F\in L^2(\Omega)$ and  $F_{\Gamma}\in L^2(\Gamma)$
$$ \dv(F): H^1(\Omega) \longrightarrow \mathbb{R},\quad u\longmapsto -\int_{\Omega} F\cdot\nabla u\, dx + \langle F\cdot \nu,u_{|\Gamma}\rangle_{H^{-\frac{1}{2}}(\Gamma),H^{\frac{1}{2}}(\Gamma)},$$
$$ \dv_{\Gamma}(F_{\Gamma}): H^1(\Gamma) \longrightarrow \mathbb{R},\quad u_{\Gamma}\longmapsto -\int_{\Gamma} F_{\Gamma}\cdot\nabla _{\Gamma} u_{\Gamma}\, d\sigma.$$
Here, $d\sigma$ is the natural surface measure on $\Gamma$.    Viewed as linear forms on $H^1(\Omega)$ and
$H^1(\Gamma)$, $\dv(F)$ and  $\dv_{\Gamma}(F_{\Gamma})$ are   continuous. In particular, we have the following estimates.
\begin{align*}
|\langle \dv(F), u\rangle| \leq C_1  \|F\|_{L^2{(\Omega)}}\|u\|_{H^{1}(\Omega)} \quad F\in L^2(\Omega), u\in H^{1}(\Omega),\\
|\langle \dv_{\Gamma}(F_{\Gamma}), u_{\Gamma}\rangle| \leq C_2 \|F_{\Gamma}\|_{L^2{(\Gamma)}}\|u_{\Gamma}\|_{H^{1}(\Gamma)} \quad F_{\Gamma}\in L^2(\Gamma), u_{\Gamma}\in H^{1}(\Gamma),
\end{align*}
where  $C_1$ and $C_2$ are  positive constants. For $F_{\Gamma} =\nabla_{\Gamma}u_{\Gamma}, u_{\Gamma}\in H^1(\Gamma)$, we define the Laplace-Beltrami operator $\Delta_\Gamma$ as follows
\begin{equation*}
\Delta_\Gamma u_{\Gamma}= \dv_{\Gamma}(\nabla_{\Gamma}u_{\Gamma}).
\end{equation*}
Throughout this paper, we assume that the matrices $\mathcal{A}$ and $\mathcal{A}_{\Gamma}$ satisfy the following assumptions.
\begin{enumerate}
	\item [(i)] $\mathcal{A}(\cdot)=(c(\cdot)_{i,j})\in C^1(\overline{\Omega}, \mathbb{R}^{N\times N})$ and $\mathcal{A}_{\Gamma}(\cdot)=(c_{_{\Gamma}}(\cdot)_{i,j})\in C^1(\Gamma, \mathbb{R}^{N\times N})$ are  symmetric, i.e., $c(x)_{i,j}=c(x)_{j,i}$ and $c_{_{\Gamma}}(x)_{i,j}=c_{_{\Gamma}}(x)_{j,i}.$
	\item [(ii)]  $\mathcal{A}(\cdot)$ and $\mathcal{A}_{\Gamma}(\cdot)$ are uniformly elliptic, in particular there exist constants  $\alpha>0 $ and $\alpha_{_{\Gamma}}>0 $  such that
	\begin{equation}\label{el1.2}
	\langle \mathcal{A}(x)\zeta, \zeta\rangle\geq	\alpha|\zeta|^2, \quad \text{and}\,\,\langle \mathcal{A}_{\Gamma}(x_{_{\Gamma}})\zeta, \zeta\rangle\geq	\alpha_{_{\Gamma}}|\zeta|^2
	\end{equation}
	for each $x\in \overline{\Omega}$, $x_{_{\Gamma}}\in \Gamma$, and  $\zeta\in\mathbb{R}^N.$
	
\end{enumerate}
\par This paper deals with Stackelberg-Nash strategy for the null controllability of the above heat equation with dynamic boundary conditions. To be more specific,  we introduce, for $i=1,2$, the non-empty open sets $\omega_{i,d}\subset \Omega$, representing the observation domains
of the followers, and the fixed target functions $y_{i,d}\in L^2(0,T;L^2( \omega_{i,d}))$. Let us define the following main cost functional
\begin{equation}
J(f)=\frac{1}{2}\int_{\omega\times (0,T)} |f|^2dx\, dt,
\end{equation}
and the secondary cost functional
\begin{equation}\label{Ji.12}
J_i(f;v_1,v_2)=\frac{\alpha_i}{2}\int_{\omega_{i,d}\times (0,T)} |y-y_{i,d}|^2dx\, dt + \frac{\mu_i}{2}\int_{\omega_{i}\times (0,T)} |v_i|^2dx\, dt,\quad i=1,2,
\end{equation}
where $\alpha_i$, $\mu_i$  are positive constants and $Y(f,v_1,v_2)=(y(f,v_1,v_2),y_{\Gamma}(f,v_1,v_2))$ is the solution to \eqref{Sys1.1SN}. For a fixed $f\in L^2(0,T; L^2(\omega))$, the pair $(v^{\star}_1,v^{\star}_2)=(v^{\star}_1(f),v^{\star}_2(f))$ is called a Nash equilibrium for $(J_1,J_2)$ associated to $f$  if
\begin{equation}
\left\{\begin{array}{ll}
{J_1(f; v^{\star}_1,v^{\star}_2)\leq J_1(f; v,v^{\star}_2) \quad \forall \, v\in  L^2(0,T;L^2( \omega_1))}, \\
{J_2(f; v^{\star}_1,v^{\star}_2)\leq J_2(f; v^{\star}_1,w) \quad \forall \, w\in  L^2(0,T;L^2( \omega_2)),}
\end{array}\right.
\end{equation}
or equivalently
\begin{equation}
\left\{\begin{array}{ll}
{J_1(f; v^{\star}_1,v^{\star}_2)= \min\limits_{v\in  L^2(0,T;L^2( \omega_1))}J_1(f; v,v^{\star}_2),} \\
{J_2(f; v^{\star}_1,v^{\star}_2)= \min\limits_{v\in  L^2(0,T;L^2( \omega_2))}J_2(f; v^{\star}_1,v).}
\end{array}\right.
\end{equation}
Since the functionals $J_i$  $(i=1,2)$, are differentiable and  convex, then, by well-known characterization results, see, e.g. Theorem 3.8 of
\cite{Jahn06} or \cite{Ek'74}, the pair $(v^{\star}_1,v^{\star}_2)=(v^{\star}_1(f),v^{\star}_2(f))$ is a Nash equilibrium for $(J_1,J_2)$ associated to $f$  if and only if
\begin{equation}\label{NE7}
\left\{\begin{array}{ll}
{\frac{\partial J_1}{\partial v_1}(f; v^{\star}_1,v^{\star}_2)(v)=0 \quad \forall \, v\in  L^2(0,T;L^2( \omega_1)),} \\
{\frac{\partial J_2}{\partial v_2}(f; v^{\star}_1,v^{\star}_2)(w)=0 \quad \forall \, w\in  L^2(0,T;L^2( \omega_2)).}
\end{array}\right.
\end{equation}

\par Our goal is to prove that, for any initial data  $Y_0=(y_0, y_{0,\Gamma})\in \mathbb{L}^2= L^2(\Omega)\times L^2(\Gamma)$,  there exist a control  $f\in  L^2(0,T;L^2( \omega))$ with minimal norm (called leader) and an associated Nash equilibrium   $(v^{\star}_1,v^{\star}_2)=(v^{\star}_1(f),v^{\star}_2(f)) \in   L^2(0,T;L^2( \omega_1))\times L^2(0,T;L^2( \omega_2))$ (called followers) such that the associated state $Y(Y_0, f, v^{\star}_1,v^{\star}_2)$ of \eqref{Sys1.1SN} satisfies $Y(f, v^{\star}_1,v^{\star}_2)(T)=0$. To do this, we shall follow  the Stackelberg-Nash strategy: for each choice of the leader $f$, we look for a Nash equilibrium pair for the cost functionals $J_i$  ($i = 1, 2$); which means finding the  controls $v^{\star}_1(f)\in L^2(0, T; L^2(\omega_1)) $ and $v^{\star}_2(f)\in L^2(0, T; L^2(\omega_2)) $, depending on $f$, satisfying  \eqref{NE7}. Once the Nash equilibrium has been identified and fixed for each $f$, we look for a control $\hat{f}$ such that
\begin{equation}\label{eq1.6}
J(\hat{f})= \min_{f\in L^2(0,T;L^2(\omega))} J(f, v^{\star}_1(f), v^{\star}_2(f)),
\end{equation}
subject to  the null controllability constraint
\begin{equation}\label{eq1.7}
Y(T, \hat{f}, v^{\star}_1(\hat{f}), v^{\star}_2(\hat{f}))(T)=0.
\end{equation}
Assume that the control regions  satisfy the following assumption
\begin{equation}\label{Assump10}
\omega_{d}=\omega_{1,d}=\omega_{2,d}, \quad \text{and }\,\omega_d\cap \omega\neq\emptyset.
\end{equation}
The main result in this paper is the following.

\begin{theorem}\label{th4.1SN}
	Let us assume that \eqref{Assump10} holds and $\mu_i>0$, $i=1,2$,  are sufficiently large. There exists a positive  weight function $\rho =\rho(t)$ blowing up at $t=T$ such that for every $y_{i,d}\in L^2(0,T; L^2(\omega_i))$ satisfying
	\begin{equation}\label{Ass11SN}
	\int_{\omega_{i, d}\times(0, T)}\rho^2 y_{i,d}^2 dx  \,dt < \infty,\quad i=1,2,
	\end{equation}
	and  every $Y_0=(y_0,y_{\Gamma, 0} )\in\mathbb{L}^2$,  there  exist a control $\hat{f}\in L^2(0,T;L^2(\omega))$ with minimal norm and an associated Nash equilibrium $(v^{\star}_1,v^{\star}_2)\in L^2(0,T;L^2(\omega_1))\times L^2(0,T;L^2(\omega_2))$
	such that the corresponding solution to \eqref{Sys1.1SN}   satisfies  $Y(T, \hat{f}, v^{\star}_1, v^{\star}_2)=0$.
\end{theorem}
\begin{remark}
	
	(a) As mentioned by several authors, the assumption \eqref{Ass11SN} seems to be natural. It means that the follower targets  $y_{i,d}$, $i=1,2$, approach to 0 as $T\longrightarrow0$, and with this  the leader  finds no obstruction to control the system.  It remains an open problem to verify if this condition is
	necessary.\\
	(b)  In \cite{ArFeGu17}, the authors have considered some  weak conditions than \eqref{Assump10}, for instance $\omega_{1, d}\cap\omega\neq\omega_{2, d}\cap\omega$, in the context of the heat equation with Dirichlet conditions. This generalization  will be treated in a forthcoming paper for dynamic boundary conditions.\\
	(c) The hierarchical control is  motivated by applications where more than
	one objective is desirable. For example, we can think of $y$ and $y_{\Gamma}$ the concentration of some chemical product in $\Omega$ and on the boundary $\Gamma$, respectively. The
	process is to guide the system under consideration to $0$ by means of an optimal control $f$ acting on the domain $\omega$, and without, in the course of the action, going too far from $y_d$ in a small subdomain $w_d$. 	
\end{remark}

 \section{Preliminaries }\label{S.2}
 Let $\Omega\subset \mathbb{R}^N$ be a bounded open set with smooth boundary $\Gamma :=\partial\Omega$. Following \cite{MMS17}, we introduce the product space defined by
 $$ \mathbb{L}^2 = L^2(\Omega)\times L^2(\Gamma), \qquad .$$
 Here,  we have considered the Lebesgue measure $dx$ on $\Omega$ and the natural surface measure $d\sigma$ on $\Gamma.$ Equipped by the scalar product
 $$ \big\langle (u,w),(v,z)\big\rangle_{\mathbb{L}^2}= \langle u,v\rangle_{L^2(\Omega)} + \langle w,z\rangle_{L^2(\Gamma)} = \int_{\Omega}u v \,dx +  \int_{\Gamma}w z \,d\sigma.$$
 $\mathbb{L}^2$ is a Hilbert space.
 Recall that $H^1(\Gamma)$ and $H^2(\Gamma)$ are real Hilbert spaces endowed with the respective norms
 $$ \|u\|_{H^1(\Gamma)}= \langle u,u \rangle^{\frac{	1}{2}}_{H^1(\Gamma)}, \,\text{with} \,\,\langle u,v \rangle_{H^1(\Gamma)}= \int_{\Gamma}u v d\sigma + \int_{\Gamma}\nabla_{\Gamma} u \nabla_{\Gamma}v d\sigma, \quad $$
 and
 $$ \|u\|_{H^2(\Gamma)}= \langle u,u \rangle^{\frac{	1}{2}}_{H^2(\Gamma)},\, \text{with}\,\, \langle u,v \rangle_{H^2(\Gamma)}= \int_{\Gamma}u v \,d\sigma +\int_{\Gamma}\Delta_{\Gamma} u \Delta_{\Gamma}v d\sigma.$$
 We point out that the operator $\Delta_{\Gamma}$  can be considered as an unbounded linear operator from $L^2(\Gamma)$ in $L^2(\Gamma)$, with domain
 $$D(\Delta_{\Gamma}) =\{ u\in L^2(\Gamma) : \,\, \Delta_{\Gamma}u\in L^2(\Gamma)\},$$
 and  it is known   that $-\Delta_{\Gamma}$ is  a self-adjoint and nonnegative operator on $L^2(\Gamma)$. This implies that  $-\Delta_{\Gamma}$ generates an analytic  $C_0-$semigroup $(e^{t\Delta_\Gamma})_{t\geq 0}$  on $L^2(\Gamma)$. If $\Gamma$ is
 smooth, then one can show that $D(\Delta_\Gamma) = H^2(\Gamma)$, and $u\mapsto \|u\|_{L^2(\Gamma)}+ \|\Delta_{\Gamma}u\|_{L^2(\Gamma)}$ defines an equivalent norm on $H^2(\Gamma)$, see {\cite{MMS17}}  and  the references therein for more details.
 As in {\cite{MMS17}}, we denote
 $$\mathbb{H}^k=\{(u,u_\Gamma)\in H^k(\Omega)\times H^k(\Gamma): \,\, u_\Gamma=u_{|\Gamma}\},\quad k=1,2,$$
 viewed as a subspace of $H^k(\Omega)\times H^k(\Gamma)$ with the natural topology inherited by $H^k(\Omega)\times H^k(\Gamma)$ and
 $$ \mathbb{E}(t_0, t_1)= H^2(t_0,t_1; \mathbb{L}^2) \cap L^2(t_0,t_1; \mathbb{H}^2) \quad \text{for} \,\,t_1>t_0,  \, \text{and} \,\,\mathbb{E}_1=\mathbb{E}(0, T).$$
 We denote $(H^k(\Omega))'$, $H^{-k}(\Gamma)$ and  $\mathbb{H}^{-k}$ the topological dual of $H^k(\Omega)$, $H^{k}(\Gamma)$ and  $\mathbb{H}^{k}$, respectively, $k=1,2$, and
 $$	\mathbb{W}=\{U\in L^2(0,T; \mathbb{H}^1): \,\,\, U^{\prime}\in L^2(0,T; \mathbb{H}^{-1}) \}.$$
 Now we shall recall some results  on  the  well-posedness of the nonhomogeneous  forward    system
 \begin{equation}
 \left\{\begin{array}{ll}
 {\partial_t y-\text{div}(\mathcal{A}\nabla y)  +B(x,t)\cdot \nabla y+ a(x,t)y= f }& {\text { in } \Omega_T,} \\
 {\partial_t y_{_{\Gamma}}   - \text{div}_{\Gamma}(\mathcal{A}_{\Gamma}\nabla_{\Gamma} y_{\Gamma})    + \partial_{\nu}^{\mathcal{A}}y+B_{\Gamma}(x, t)\cdot \nabla_{\Gamma} y_{\Gamma} + b(x,t)y_{\Gamma}=g} &\,\,{\text {on} \,\Gamma_T,} \\
 {(y(0),y_{\Gamma}(0)) =(y_0,y_{\Gamma,0})} & {\text { in } \Omega_T} \label{sys2.2}
 \end{array}\right.
 \end{equation}
 and the nonhomogeneous backward one
{\small \begin{equation}
 \left\{\begin{array}{ll}
 {-\partial_t \varphi-\text{div}(\mathcal{A}\nabla \varphi)  -\text{div}(B\varphi)+ a(x,t)\varphi=f_1}& {\text {in} \,\Omega_T,} \\
 {-\partial_t \varphi_{\Gamma}   -\text{div}_{\Gamma}(\mathcal{A}_{\Gamma}\nabla_{\Gamma} \varphi_{\Gamma})     + \partial^{\mathcal{A}}_{\nu}\varphi - \text{div}_{\Gamma}( B_{\Gamma}\varphi_{\Gamma})+\varphi_{\Gamma} B\cdot\nu + b(x,t)\varphi_{\Gamma}=g_1,}&\,{\text{on}\,\Gamma_T,} \\
 {(\varphi(T),\varphi_{\Gamma}(T) ) =(\varphi_T,\varphi_{\Gamma,T}),} &\,{\text {in}\, \Omega \times\Gamma} \label{systadjoint2.3}.
 \end{array}\right.
 \end{equation}}Remark first that the system \eqref{sys2.2} can be rewritten as the following
 abstract Cauchy problem
 \begin{equation}\label{1.8}
 \left\{
 \begin{array}{ll}
 Y'(t) = A\,Y - D(t)\,Y+F \quad  t>0, \\

 Y(0)=Y_0= (y_0,y_{\Gamma,0}),
 \end{array}
 \right.
 \end{equation}
 where we  denoted
 $Y=(y, y_\Gamma), F=(f, g),$
 $$A=\begin{pmatrix}
 \mathrm{div}(\mathcal{A}\nabla )  & 0 \\
 -\partial^{\mathcal{A}}_{\nu} & \mathrm{div}_{\Gamma}(\mathcal{A}_{\Gamma}\nabla_{\Gamma})
 \end{pmatrix}, \mathcal{D}(A) =\mathbb{H}^2$$
 and $$D(t)=\begin{pmatrix}
 B(t)\cdot\nabla + a(t) & 0 \\
 0 & B_{\Gamma}(t)\cdot\nabla_{\Gamma} + b(t)
 \end{pmatrix}.$$
 Following \cite{MMS17}, we can show that the operator A satisfies the following important property.
 \begin{proposition}[\cite{MMS17}]\label{t4.2}
 	The operator $A$ is densely defined, and generates an analytic $C_0$-semigroup $(e^{tA})_{t\geq 0}$ on $\mathbb{L}^2.$ We have also  $(\mathbb{L}^2, \mathbb{H}^2 )_{\frac{1}{2}, 2}= \mathbb{H}^1$.
 \end{proposition}
 The following existence and uniqueness results hold, see \cite{BCMO20} for the proof.
 \begin{proposition} [\cite{BCMO20}]\label{t4.22}
 	For every $Y_0=(y_0,y_{\Gamma,0})\in\mathbb{L}^2$, $f\in L^2(\Omega_T)$ and  $g\in L^2(\Gamma_T)$ the system \eqref{sys2.2} has a unique mild solution given by \begin{equation}\label{4.2}
 	\small
 	Y(t)= e^{tA}Y_0 \,+\, \int_{0}^{t}e^{(t-s)A}(F(s)- D(s)Y(s))ds
 	\end{equation}
 	for all $t\in [0,T]$.
 	Moreover,  there exists a
 	constant $C > 0$ such that
 	\begin{equation}\label{4.3}
 	\|Y\|_{C([0,T];\mathbb{L}^2)}\leq C\big( \|Y_0\|_{\mathbb{L}^2} + \|f\|_{L^2(\Omega_T) } + \|g\|_{L^2(\Gamma_T) } \big).
 	\end{equation}
 \end{proposition}
 For the backward system, we have the following well-posedness result, see \cite{KhMaMaGh19}.
 \begin{proposition}\label{prop2.3}
 	For every $\Phi_T=(\varphi_T,\varphi_{\Gamma,T})\in\mathbb{L}^2$, $f_1\in L^2(0,T; (H^1(\Omega))')$ and $g_1\in L^2(0,T; H^{-1}(\Gamma))$, the backward system \eqref{systadjoint2.3} has a unique weak solution $\Phi=(\varphi, \varphi_{\Gamma})\in \mathbb{W}$. That is
 	\begin{align}
 	&\int_{0}^{T}\langle \partial_t\varphi, v\rangle_{(H^1(\Omega))^{'}, H^1(\Omega) }\,dt + \int_{\Omega_T}\mathcal{A}\nabla\varphi\cdot\nabla v dx\,dt
 	- \int_{\Omega_T}\varphi B\cdot\nabla v dx\,dt\nonumber\\
 	&+\int_{\Omega_T}a \varphi v dx\,dt + \int_{0}^{T}\langle \partial_t\varphi_{\Gamma}, v_{\Gamma}\rangle_{H^{-1}(\Gamma), H^1(\Gamma) }\,dt +\int_{\Gamma_T}\mathcal{A}_{\Gamma}\nabla_{\Gamma}\varphi_{\Gamma}\cdot\nabla_{\Gamma} v_{\Gamma} d\sigma\,dt \nonumber\\
 	& -\int_{\Gamma_T}\varphi_{\Gamma} B_{\Gamma}\cdot\nabla_{\Gamma} v_{\Gamma} d\sigma\,dt +\int_{\Gamma_T}b\varphi_{\Gamma} v_{\Gamma} d\sigma\,dt=  \int_{0}^{T}\langle f_1, v\rangle_{(H^1(\Omega))^{'}, H^1(\Omega) }\,dt\nonumber\\
 	&+\int_{0}^{T}\langle g_1, v_{\Gamma}\rangle_{H^{-1}(\Gamma), H^1(\Gamma) }\,dt\label{Kh1.12}
 	\end{align}
 	for each $(v,v_{\Gamma})\in L^2(0,T; \mathbb{H}^1)$, with  $v(0)=v_{\Gamma}(0)=0$ and $(\varphi(T), \varphi_{\Gamma}(T) )=(\varphi_T, \varphi_{\Gamma, T}).$
 	Moreover, we have the estimate
 	\begin{equation*}
 	\max_{0\leq t\leq T}\|\Phi(t)\|^2_{\mathbb{L}^2}+\|\Phi\|^2_{L^2(0,T ; \mathbb{H}^1)}+\|\Phi^{\prime}\|^2_{L^2(0,T; \mathbb{H}^{-1})}\leq
 	\end{equation*}
 	\begin{equation}\label{Eq.19NS}
 	C\big(\|\Phi_T\|^2_{\mathbb{L}^2}+ \|f_1\|^2_{L^2(0,T; (H(\Omega)^{1})')} +\|g_1\|^2_{L^2(0,T; H^{-1}(\Gamma))}),
 	\end{equation}
 	where  $C$ is  a positive constant.
 \end{proposition}	

 \section{{ Existence, uniqueness and characterization of Nash-equilibrium.}}\label{S.3}
 In this section, using the same ideas as in \cite{ArCaSa15}, we shall prove the existence and provide a characterization, in term of an adjoint system,  of the Nash-equilibrium in the sense  of \eqref{NE7}. In the sequel, if $X$ is a Hilbert space, $\langle \cdot, \cdot \rangle _{X}$ stands for the scalar product  of $X$.

 \par For $i=1,2$, consider  the functional given by \eqref{Ji.12} and denote  the spaces
 {\small\begin{equation*}
 \mathcal{H}_i=L^2(0,T;L^2(w_i)),\,\, \mathcal{H}_{i,d}=L^2(0,T;L^2(w_{i,d})),   \,\,\mathcal{H}=\mathcal{H}_1\times \mathcal{H}_2\,\,\text{and} \,\, \mathcal{H}_d=\mathcal{H}_{1,d}\times \mathcal{H}_{2,d}.
 \end{equation*}}Define the operators $L_i\in \mathcal{L}(\mathcal{H}_i, L^2(0,T; \mathbb{L}^2))$ and  $\ell_i\in \mathcal{L}(\mathcal{H}_i, L^2(0,T; L^2(\Omega)))$  by
 \begin{align*}
 L_i(v_i)=Y_i= (y_i, y_{\Gamma,i}),\,\,\text{and}\,\, \ell_i(v_i)=y_i,
 \end{align*}
 where $Y_i=(y_i, y_{\Gamma,i})$ is the solution to
 {\small\begin{equation}\label{Eq20NS}
 \left\{\begin{array}{ll}
 {\partial_t y_i-\mathrm{div}(\mathcal{A}\nabla y_i)  +B(x,t)\cdot \nabla y_i+ a(x,t)y_i= v_i1_{\omega_i}}& {\text { in } \Omega_{T},} \\
 {\partial_t y_{_{\Gamma,i}}   - \mathrm{div}_{\Gamma}(\mathcal{A}_{\Gamma}\nabla_{\Gamma} y_{\Gamma,i})    + \partial_{\nu}^{\mathcal{A}}y_i+B_{\Gamma}(x, t)\cdot \nabla_{\Gamma} y_{\Gamma,i} + b(x,t)y_{\Gamma, i}=0}& {\text { on } \Gamma_{T},} \\
 {(y_i(0),y_{\Gamma,i}(0)) =0}& {\text { in } \Omega \times \Gamma}.
 \end{array}\right.
 \end{equation}}So, we can write $Y$ as follows  $$Y(f,v_1,v_2)=L_1(v_1)+L_2(v_2)+Q(f),$$
 where $Q(f)=(q, q_{\Gamma})$ solves the system
 \begin{equation}
 \left\{\begin{array}{ll}
 {\partial_t q-\mathrm{div}(\mathcal{A}\nabla q)  +B(x,t)\cdot \nabla q+ a(x,t)q= 1_{\omega}f}& {\text { in } \Omega_{T},} \\[2mm]
 {\partial_t q_{\Gamma}   - \mathrm{div}_{\Gamma}(\mathcal{A}_{\Gamma}\nabla_{\Gamma} q_{\Gamma})    + \partial_{\nu}^{\mathcal{A}}q_{\Gamma}+B_{\Gamma}(x, t)\cdot \nabla_{\Gamma} q_{\Gamma} + b(x,t)q_{\Gamma}=0}& {\text { on } \Gamma_{T},} \\[2mm]
 {(q(0),q_{\Gamma}(0)) =Y_0}& {\text { in } \Omega \times \Gamma}.
 \end{array}\right.
 \end{equation}
 %
 \par Using  some ideas of \cite{ArCaSa15}, we shall prove the existence and uniqueness of Nash-equilibrium. More precisely, we have the following result.
 \begin{proposition}\label{Prop04}
 	There exists a constant $\mu_0 >0$ such that, if $\mu_i\geq \mu_0$, $i = 1, 2$, then for   each   $f\in L^2(0,T;L^2(\omega))$ there exists a unique Nash-equilibrium $(v^{\star}_1(f),v^{\star}_2(f))\in \mathcal{H}_1\times \mathcal{H}_2$ for $(J_1, J_2)$
 	associated to $f$. Furthermore, there exists a constant $C>0 $  such that
 	\begin{equation}
 	\|(v^{\star}_1(f),v^{\star}_2(f))\|_{\mathcal{H}}\leq C(1+ \|f\|_{L^2(0,T;L^2(\omega))}).
 	\end{equation}
 \end{proposition}
 \begin{proof}
 	For $i=1,2$ and a fixed $f$, we have 	
 	{\small\begin{equation*}
 	\partial_i J_i(f; v^{\star}_1,v^{\star}_2)(v_i)= \alpha_i \langle \ell_1(v^{\star}_1)+\ell_2(v^{\star}_2) +q(f)-y_{i,d},	\ell_i(v_i)\rangle_{\mathcal{H}_{i,d}} + \mu_i \langle v^{\star}_i,v_i\rangle_{\mathcal{H}_{i}}\quad \forall v_i\in \mathcal{H}_{i},
 	\end{equation*}}here, $\partial_i J_i=\frac{\partial J_i}{\partial v_i}$ stands for the $i$-th partial derivative of $J_i.$ Using \eqref{NE7}, we deduce that
 	$(v^{\star}_1,v^{\star}_2)$ is a Nash-equilibrium if and only if
 	$${\alpha_i \langle \ell_1(v^{\star}_1)+\ell_2(v^{\star}_2) -\tilde{y}_{i,d},\ell_i(v_i)\rangle_{\mathcal{H}_{i,d}} + \mu_i \langle v^{\star}_i,v_i\rangle_{\mathcal{H}_{i}}}=0 \quad \forall v_i\in \mathcal{H}_{i},$$
 	where $\tilde{y}_{i,d}=y_{i,d}-q$. Thus
 	\begin{equation*}
 	\alpha_i \langle \ell_i^{\star}[\ell_1(v^{\star}_1)+\ell_2(v^{\star}_2)-\tilde{y}_{i,d}],	v_i\rangle_{\mathcal{H}_{i,d}} + \mu_i \langle v^{\star}_i,v_i\rangle_{\mathcal{H}_{i}}=0\quad \forall v_i\in \mathcal{H}_{i}.
 	\end{equation*}
 	Finally, $(v^{\star}_1,v^{\star}_2)$ is a Nash-equilibrium if and only if
 	\begin{equation*}\label{CARA1}
 	\alpha_i  \ell_i^{\star}[\ell_1(v^{\star}_1)1_{w_{i,d}}+ \ell_2(v^{\star}_2)1_{w_{i,d}}] + \mu_i  v^{\star}_i=\alpha_i\ell_i^{\star}(\tilde{y}_{i,d}1_{w_{i,d}}).
 	\end{equation*}
 	We define the operator $R=(R_1, R_2)\in\mathcal{L}(\mathcal{H})$ as follows
 	\begin{equation*}
 	R_i(v^{\star}_1,v^{\star}_2)=	\alpha_i  \ell_i^{\star}[\ell_1(v^{\star}_1)1_{w_{i,d}} +\ell_2(v^{\star}_2)1_{w_{i,d}}] + \mu_i  v^{\star}_i.
 	\end{equation*}
 	We have
 	$$R(v^{\star}_1,v^{\star}_1)= (\alpha_i\ell_1^{\star}(\tilde{y}_{1,d}1_{w_{1,d}}),\alpha_2\ell_1^{\star}(\tilde{y}_{2,d}1_{w_{2,d}})).$$
 	Let us prove that $R$ is invertible. We have	
 	\begin{align*}
 	\langle R(v^{\star}_1,v^{\star}_1),(v^{\star}_1,v^{\star}_1)\rangle_{\mathcal{H}} &=	  \langle R_1(v^{\star}_2,v^{\star}_1),v^{\star}_1\rangle_{\mathcal{H}_1}+ \langle R_2(v^{\star}_1,v^{\star}_2),v^{\star}_2\rangle_{\mathcal{H}_2}\\[2mm]
 	&=\alpha_1\langle   1_{w_{1,d}}\ell_1(v^{\star}_1)+ 1_{w_{1,d}}\ell_2(v^{\star}_2) ,\ell_1(v^{\star}_1)\rangle_{\mathcal{H}_1}+ \mu_1\|  v^{\star}_1\|^2_{\mathcal{H}_1}\\[2mm]
 	& \;+\alpha_2\langle   1_{w_{2,d}}\ell_1(v^{\star}_1)+ 1_{w_{2,d}}\ell_2(v^{\star}_2) ,\ell_2(v^{\star}_1)\rangle_{\mathcal{H}_2}+ \mu_2\|  v^{\star}_2\|^2_{\mathcal{H}_2}\\[2mm]
 	&=\alpha_1\langle   1_{w_{1,d}}\ell_2(v^{\star}_2) ,\ell_1(v^{\star}_1)\rangle_{\mathcal{H}_1} +\alpha_1\| 1_{w_{1,d}}\ell_1(v^{\star}_1)\|^2_{_{\mathcal{H}_1}} +\mu_1\|  v^{\star}_1\|^2_{\mathcal{H}_1}\\[2mm]
 	&\; +\alpha_2\langle    1_{w_{2,d}}\ell_1(v^{\star}_1) ,\ell_2(v^{\star}_2)\rangle_{\mathcal{H}_2} + \alpha_1\| 1_{w_{2,d}}\ell_2(v^{\star}_2)\|^2_{_{\mathcal{H}_2}} +\mu_2\|  v^{\star}_2\|^2_{\mathcal{H}_2}\\[2mm]
 	&\geq -\frac{\alpha_1}{4}\|   1_{w_{1,d}}\ell_2 \|^2_{\mathcal{L}(\mathcal{H}_2,\mathcal{H}_{1,d})}\|  v^{\star}_2\|^2_{\mathcal{H}_2}  +\mu_1\|  v^{\star}_1\|^2_{\mathcal{H}_1}\\[2mm]
 	&\; -\frac{\alpha_2}{4}\|   1_{w_{2,d}}\ell_1 \|^2_{\mathcal{L}(\mathcal{H}_1,\mathcal{H}_{2,d})}\|  v^{\star}_1\|^2_{\mathcal{H}_1}  +\mu_2\|  v^{\star}_2\|^2_{\mathcal{H}_2}.
 	\end{align*}
 	Choosing the parameters $\mu_1$ and $\mu_2$ such that
 	\begin{align*}
 	4 \mu_{1}>\alpha_{2}\|   1_{w_{2,d}}\ell_1 \|^2_{\mathcal{L}(\mathcal{H}_1,\mathcal{H}_{2,d})} \,\,\text{and}\,\,
 	4 \mu_{2}>\alpha_{1}\|   1_{w_{1,d}}\ell_2 \|^2_{\mathcal{L}(\mathcal{H}_2,\mathcal{H}_{1,d})},
 	\end{align*}
 	there exists a constant $C_1>0$, such that
 	\begin{align*}
 	\langle R(v^{\star}_1,v^{\star}_1),(v^{\star}_1,v^{\star}_1)\rangle_{\mathcal{H}} \geq C_1 \|(v^{\star}_1, v^{\star}_2)\|^2_{\mathcal{H}}.
 	\end{align*} 	
 	Since $R$ is continuous, we deduce, from  Lax-Milgramm theorem,  that $R$ invertible. Furthermore, we have
 	\begin{equation*}
 	\|(v^{\star}_2, v^{\star}_2)\|_{\mathcal{H}}\leq C\|(\alpha_1\ell_1^{\star}(\tilde{y}_{1,d}1_{w_{1,d}}),\alpha_2\ell_1^{\star}(\tilde{y}_{2,d}1_{w_{2,d}}))\|_{\mathcal{H}}\leq C(1+ \|f\|_{L^2(0,T;L^2(\omega)}).
 	\end{equation*}
 	This achieves the proof.
 \end{proof}	
 Now, we shall characterize the Nash-equilibrium in term of the solution to an adjoint system. To this end,  let us introduce the following adjoint systems
 {\small\begin{equation}\label{Eq25S}
 	\left\{\begin{array}{ll}
 	{-\partial_t \varphi^i-\mathrm{div}(\mathcal{A}\nabla \varphi^i)- \mathrm{div}(\varphi^i B) + a(x,t)\varphi^i=\alpha_i(y-y_{i,d})1_{\omega_{i,d}} }& {\text { in } \Omega_{T},} \\[2mm]
 	{-\partial_t \varphi^i_{_{\Gamma}}   -\mathrm{div}_{\Gamma}(\mathcal{A}_{\Gamma}\nabla_{\Gamma} \varphi_{\Gamma}^i)- \mathrm{div}_{\Gamma}(\varphi_{\Gamma}^iB_{\Gamma})+ \varphi^i_{\Gamma} B\cdot \nu + \partial^{\mathcal{A}}_{\nu}\varphi^i + b(x,t)\varphi^i_{\Gamma}=0}& {\text { on } \Gamma_{T},} \\[2mm]
 	{(\varphi^i(T),\varphi^i_{\Gamma}(T)) =0}& {\text { in } \Omega \times \Gamma,}\\[2mm]
 	i=1,2.
 	\end{array}\right.
 	\end{equation}}Multiplying  \eqref{Eq20NS} by $\Phi^i=(\varphi^i,\varphi^i_{\Gamma})$ and integrating by parts, we find
 \begin{equation}
 \alpha_i\langle y(f,v_1,v_2), \ell_i(v_i)\rangle_{L^2(0,T; L^2(\Omega))}= \langle \varphi, v_i\rangle_{\mathcal{H}_i}.
 \end{equation}
 We deduce that, $(v^{\star}_1,v^{\star}_2)$ is a Nash-equilibrium if and only if
 $$\langle \varphi^1, v_1\rangle_{\mathcal{H}_1}+\mu_1\langle v^{\star}_1, v_1\rangle_{\mathcal{H}_1}=0\quad \text{and}\,\, \langle \varphi^2, v_2\rangle_{\mathcal{H}_2}+\mu_2\langle v^{\star}_2, v_2\rangle_{\mathcal{H}_2}=0, \quad  (v_1, v_2)\in \mathcal{H}.$$
 Then
 \begin{equation}\label{chara.1}
 v^{\star}_i=-\frac{1}{\mu_i} \varphi^i|_{\omega_i\times(0,T)} \quad \text{for} \,\,i=1,2.
 \end{equation}
 Let us collect all the previous  results in a same system. We obtain the following forward–backward system, called optimality system
 {\small\begin{equation}  \label{Sys.28NS}
 \left\{\begin{array}{ll}
 {\partial_t y-\mathrm{div}(\mathcal{A}\nabla y)  +B(x,t)\cdot \nabla y+ a(x,t)y=f1_{\omega} -\frac{1}{\mu_1} \varphi^1 1_{\omega_1}-\frac{1}{\mu_2} \varphi^2 1_{\omega_2} }& {\text { in } \Omega_{T},} \\[2mm]
 {	\partial_t y_{_{\Gamma}}   -\mathrm{div}_{\Gamma}(\mathcal{A}_{\Gamma}\nabla_{\Gamma} y_{\Gamma})  +B_{\Gamma}(x,t)\cdot \nabla_{\Gamma} y_{\Gamma} + \partial^{\mathcal{A}}_{\nu}y + b(x,t)y_{\Gamma}=0}& {\text { on } \Gamma_{T},} \\[2mm]
 {	-\partial_t \varphi^i-\mathrm{div}(\mathcal{A}\nabla \varphi^i)  - \mathrm{div}( \varphi^i B) + a(x,t)\varphi^i=\alpha_i(y-y_{i,d})1_{\omega_{i,d}}}&\,\, {\text {in}\,\Omega_{T},}\\[2mm]
 {	-\partial_t \varphi^i_{_{\Gamma}}   - \mathrm{div}_{\Gamma}(\mathcal{A}_{\Gamma}\nabla_{\Gamma} \varphi^i_{\Gamma})-\mathrm{div}_{\Gamma}( \varphi^i_{\Gamma}B_{\Gamma})+ \varphi^i_{\Gamma} B\cdot\nu  + \partial^{\mathcal{A}}_{\nu}\varphi^i + b(x,t)\varphi^i_{\Gamma}=0} & {\text { on }  \Gamma_T,}\\[2mm]
 {(y(0),y_{\Gamma}(0))=Y_0}& {\text { in}\,\, \Omega \times \Gamma,} \\[2mm]
 {(\varphi^i(T),\varphi^i_{\Gamma}(T)) =0}& {\text { in } \Omega \times \Gamma,}\\[2mm]
 i=1,2.
 \end{array}\right.
 \end{equation}}In Section \ref{S.4}, we will show that the null controllability  for the system \eqref{Sys.28NS} is equivalent to a suitable  observability  inequality  for  the following  adjoint system
 {\small\begin{equation}\label{Sys.29NS}
 \left\{\begin{array}{ll}
 {-\partial_t z- \mathrm{div}(\mathcal{A}\nabla z)- \mathrm{div}(B z)+ a(x,t)z=\alpha_1 \psi^1 1_{\omega_{1, d}}+\alpha_2 \psi^2 1_{\omega_{1,d}} }& {\text { in } \Omega_{T},} \\[2mm]
 {-\partial_t z_{_{\Gamma}}   -\mathrm{div}_{\Gamma}(\mathcal{A}_{\Gamma}\nabla_{\Gamma} z_{\Gamma})- \mathrm{div}_{\Gamma}(B_{\Gamma}z_{\Gamma})+ z_{\Gamma}B\cdot\nu   + \partial^{\mathcal{A}}_{\nu}z+ b(x,t)z_{\Gamma}=0}& {\text { on } \Gamma_{T},} \\[2mm]
 {	\partial_t \psi^i- \mathrm{div}(\mathcal{A}\nabla \psi^i)+ B\cdot\nabla\psi^i + a(x,t)\psi^i=-\frac{1}{\mu_i}z\,1_{\omega_{i}}} & \,\,{\text {in } \,\,\Omega_{T},}\\[2mm]
 {	\partial_t \psi^i_{_{\Gamma}}   -\mathrm{div}_{\Gamma}(\mathcal{A}_{\Gamma}\nabla_{\Gamma} \psi^i_{\Gamma})+ B_{\Gamma}\cdot\nabla\psi^i_{\Gamma}  + \partial^{\mathcal{A}}_{\nu}\psi^i + b(x,t)\psi^i_{\Gamma}=0} & {\text { on } \Gamma_{T},}\\[2mm]
 {(z(T),z_{\Gamma}(T))=Z_T}& {\text { in }\Omega \times \Gamma,} \\[2mm]
 {(\psi^i(0),\psi^i_{\Gamma}(0)) =0}& {\text { in } \Omega \times \Gamma,}\\[2mm]
 i=1,2.
 \end{array}\right.
 \end{equation}}
 \vspace*{-10pt}
 \section{Carleman estimates and  null controllability}\label{S.4}
 \subsection{Carleman estimates}
 In this section, we shall state and   show some suitable Carleman estimates needed to prove  our main result concerning null controllability. To this end,   let us first  introduce the  following  well-known Morse function, see {\cite{B5}} for the existence of such  function.
 	Let  $\omega'$  be an open set of $\Omega$  such that
 	\begin{equation*}
 	\omega'\subset\omega\cap \omega_d
 	\end{equation*} and $\eta_0\in C^2(\overline{\Omega})$ be a function  such that
 	\begin{equation*}
 	\left\{
 	\begin{array}{ll}
 	\eta_0> 0  \quad  \text{in } \Omega  \quad \text{and} \quad\eta_0=0 \quad \text{on } \Gamma,\\
 	|\nabla\eta_0|  \neq 0  \quad  \text{in }   \quad \overline{\Omega\backslash\omega' },\\
 	|\nabla_{\Gamma}\eta_0|=0, \quad \partial_{\nu}\eta_0<-c, \quad \nabla\eta_0= \partial_{\nu}\eta_0 \nu\quad \text{on} \,\,\Gamma
 	\end{array}
 	\right.
 	\end{equation*}
 	for some constant $c>0$.

 Introduce the following   classical weight functions
 $$\xi(x,t)=\frac{e^{{\lambda (2\|\eta_0\|_{\infty} + \eta_0(x)) }}}{t(T -t)}  \quad and \quad \alpha(x,t)=\frac{e^{{4\lambda  \|\eta_0\|_{\infty} }}- e^{{2\lambda (\|\eta_0\|_{\infty} + \eta_0(x))}}}{t(T -t)},$$
 where $x\in \overline{\Omega},\, t\in(0,T)$ and $\lambda \geq 1.$ The following lemma summarizes some important properties of the above  functions.
 In what follows, $C$ stands for a generic positive constant only  depending on $\Omega$  and $\omega$, whose value can change from line to line.
 \begin{lemma}\label{Lem2.5} The functions $\xi$ and $\alpha$ satisfy  the following properties.
 	\begin{enumerate}
 		\item $|\partial_t \alpha| \leq C \xi^2$,\,\, $|\partial_t \xi|\leq C \xi^2 $, \,\, $|\nabla \alpha|\leq C\lambda \xi $.
 		\item $\left|\partial_{t}\left(s^3\lambda^4\xi^{3} e^{-2 s \alpha} \right)\right| \leq C s^4\lambda^4 \xi^{5} e^{-2 s \alpha} $,\quad  $\left|\nabla\left(s^3\lambda^4\xi^3 e^{-2 s \alpha} \right)\right| \leq C s^{4} \lambda^5 \xi^{4}e^{-2 s \alpha} $.
 		\item $|\mathrm{div}\left(\mathcal{A}\nabla\left(\xi^3 e^{-2s\alpha}\right)\right)|\leq Cs^5\lambda^6 \xi^5 e^{-2s\alpha}$.
 		\item For all $s>0$ and $r \in \mathbb{R},$ the function $e^{-2 s \alpha} \xi^{r}$ is bounded on $\Omega_{T}$.		
 	\end{enumerate}
 \end{lemma}

 Now, we  recall some Carleman estimates for heat equation with dynamic boundary conditions needed to show our main result. Let us first introduce the following quantity
 \begin{align*}
 \mathcal{I}(s,\lambda; \Phi)=&  s\lambda^2\int_{\Omega_T}\xi e^{-2s\alpha}|\nabla\varphi|^2dx\,dt  +s\lambda\int_{\Gamma_T}\xi e^{-2s\alpha}|\nabla_{\Gamma}\varphi_{\Gamma}|^2 d\sigma\,dt \\
 &\quad+s\lambda\int_{\Gamma_T}\xi e^{-2s\alpha}|\partial_{\nu}\varphi|^2 d\sigma\,dt
 +s^3\lambda^4\int_{\Omega_T}\xi^3 e^{-2s\alpha}|\varphi|^2dx\,dt\\
 &\quad  \,\,+\,
 s^3\lambda^3\int_{\Gamma_T}\xi^3 e^{-2s\alpha}|\varphi_\Gamma|^2d\sigma\,dt,
 \end{align*}
 where $\lambda$ and $s$ are positive real numbers and $\Phi=(\varphi, \varphi_{\Gamma})$ is a smooth function. Consider the following general form  of the adjoint system
 \begin{equation}\label{Syst37SN}
 \left\{\begin{array}{ll}
 {-\partial_t q-\mathrm{div}(\mathcal{A}\nabla q)=F_0+ \mathrm{div}(F)  }& {\text { in } \Omega_{T},} \\
 {-\partial_t q_{_{\Gamma}}   -\mathrm{div}_{\Gamma}(\mathcal{A}_{\Gamma}\nabla_{\Gamma} q_{\Gamma})+ \partial^{\mathcal{A}}_{\nu}q = -F\cdot\nu + F_{\Gamma,0}} +\mathrm{div}_{\Gamma}(F_{\Gamma})& {\text { on } \Gamma_{T},} \\
 {(q(T),q_{\Gamma}(T)) =Q_T}& {\text { in } \Omega \times \Gamma,}
 \end{array}\right.
 \end{equation}
 where $F_0\in L^2(\Omega_T), F\in (L^2(\Omega_T))^N, F_{\Gamma,0}\in L^2(\Gamma_T),F_\Gamma\in (L^2(\Gamma_T))^N$ and $Q_T\in \mathbb{L}^2$. We have the  following Carleman estimates, see \cite{ACMO20} and \cite{BCMO20}   for the proof.

 \begin{lemma}\label{Pro6SN}
 	
 	(i)If $F=F_{\Gamma}=0$, then there exist $\lambda_1\geq 1$, $s_1\geq 1$ and $C_1=C_1(\omega,\Omega)>0$ such that the solution $Q=(q, q_{\Gamma})$ to \eqref{Syst37SN}  satisfies
 	\begin{align*}
 	\mathcal{I}(s,\lambda; Q)&\leq
 	C_1\Big(\int_{\Omega_T}e^{-2s\alpha}|F_0|^2dx\,dt
 	+ \,\,\int_{\Gamma_T}e^{-2s\alpha}|F_{\Gamma,0} |^2  d\sigma\,dt\\
 	&\quad + s^3\lambda^4\int_{\omega\times(0,T)}\xi^3 e^{-2s\alpha}|q|^2dx\,dt\Big)
 	\end{align*}
 	for all $  \lambda\geq \lambda_1$ and  $s \geq s_1.$	
 	
 	(ii) If the functions $F$ and $F_{\Gamma}$ are not necessarily zero, then there exist $\lambda_2\geq 1$, $s_2\geq 1$ and $C_2=C_2(\omega,\Omega)>0$ such that the solution $Q=(q, q_{\Gamma})$ to \eqref{Syst37SN} satisfies
 	\begin{align*}
 	\mathcal{I}(s,\lambda; Q)&\leq
 	C_2\Big( \int_{\Omega_T}e^{-2s\alpha}|F_0|^2 dx\,dt+ \int_{\Gamma_T}e^{-2s\alpha}|F_{\Gamma,0}|^2 d\sigma\,dt\\&\quad+ s^2\lambda^2\int_{\Omega_T}e^{-2s\alpha}\xi^2\|F\|^2 dx\,dt\nonumber
 	+\,\,s^2\lambda^2\int_{\Gamma_T}e^{-2s\alpha}\xi^2\|F_{\Gamma}\|^2 d\sigma\,dt\\
 	& \quad + \,\,s^3\lambda^4\int_{\omega\times(0,T)}\xi^3 e^{-2s\alpha}|q|^2dx\,dt\Big)
 	\end{align*}
 	for all $  \lambda\geq \lambda_2$ and  $s \geq s_2.$
 \end{lemma}
 Now, we shall prove a Carleman estimate for the coupled  system \eqref{Sys.29NS}.

 \begin{theorem}\label{thmC4.3}
 	Assume that \eqref{Assump10} holds. Then there exist $\lambda_3\geq 1$, $s_3\geq 1$  and  $C>0$ such that every solution $(Z, \Psi^1, \Psi^2)$  to \eqref{Sys.29NS}  satisfies
 	\begin{align}\label{Car4.13}
 	\mathcal{I}(s,\lambda; Z)+ \mathcal{I}(s,\lambda; H)  \leq C s^7\lambda^8 \int_{\omega\times(0,T)} e^{-2s\alpha} \xi^7|z|^2 \, dx\,dt
 	\end{align}
 	for all $s \geq s_3$ and $\lambda \geq \lambda_3$, where  $Z=(z, z_{\Gamma})$ and
 	$H=(h, h_{\Gamma})= \alpha_1\Psi^1+\alpha_2\Psi^2=\alpha_1(\psi^1, \psi^1_{\Gamma})+\alpha_2(\psi^2, \psi^2_{\Gamma}).$
 \end{theorem}
 \begin{proof}
 	Let  $\omega'_1$ be open sets such that
 	\begin{equation}\label{Eq31D}
 	\omega'\subset \omega'_1\subset\subset \omega\cap \omega_d.
 	\end{equation}
 	Let $\theta\in C^{\infty}(\Omega)$ such that
 	\begin{equation*}
 	0\leq\theta\leq 1,\,\, \theta =1 \, \text{in}\,\, \omega'\, \text{and}\,\,\, \text{supp}(\theta)\subset \omega'_1.
 	\end{equation*}
 	Applying Carleman estimate given in the second point of Lemma \ref{Pro6SN} to $Z$ with the  observation region $\omega'$ instead of $\omega$ and  $F_0= h 1_{\omega_d} - az$,  $F= Bz$, $F_{ \Gamma,0}=- b z_{\Gamma} $, $F_{\Gamma}=B_{\Gamma} z_{\Gamma}$, we find
 	\begin{align*}\mathcal{I}(s,\lambda; Z )&\leq C\Big(\int_{\Omega_T}e^{-2s\alpha}|h |^2  dx\,dt
 	+ \int_{\Omega_T}e^{-2s\alpha}|a z |^2  dx\,dt+ \\ &s^2\lambda^2\int_{\Omega_T}\xi^2e^{-2s\alpha}|Bz |^2  dx\,dt +\\ &s^2\lambda^2\int_{\Gamma_T}\xi^2e^{-2s\alpha}|B_{\Gamma}z_{\Gamma} |^2  d\sigma\,dt+
 	\int_{\Gamma_T}e^{-2s\alpha}| b z_{\Gamma} |^2  d\sigma\,dt+\\
 	 &s^3\lambda^4\int_{\omega'\times(0,T)}\xi^3 e^{-2s\alpha}|z|^2dx\,dt\Big)\\
 	&\leq C\Big(\int_{\Omega_T}e^{-2s\alpha}|h |^2  dx\,dt
 	+ \|a\|^2_{\infty}\int_{\Omega_T}e^{-2s\alpha}| z |^2  dx\,dt+\\ & \|B\|^2_{\infty}s^2\lambda^2\int_{\Omega_T}\xi^2e^{-2s\alpha}|z |^2  dx\,dt +\\ &\|B_{\Gamma}\|^2_{\infty}s^2\lambda^2\int_{\Gamma_T}\xi^2e^{-2s\alpha}|z_{\Gamma} |^2  d\sigma\,dt+
 	 \|b\|^2\int_{\Gamma_T}e^{-2s\alpha}|z_{\Gamma} |^2  d\sigma\,dt+\\
 	&s^3\lambda^4\int_{\omega'\times(0,T)}\xi^3 e^{-2s\alpha}|z|^2dx\,dt\Big),
 	\end{align*}
 	Choosing $s_2$ large enough, the previous inequality becomes
 	\begin{equation}\label{In60SN}
 	\mathcal{I}(s,\lambda; Z )\leq C\Big(\int_{\Omega_T}e^{-2s\alpha}|h |^2  dx\,dt
 	+ s^3\lambda^4\int_{\omega'\times(0,T)}\xi^3 e^{-2s\alpha}|z|^2dx\,dt\Big).
 	\end{equation}
 	Using the Carleman estimate given in the first point of Lemma  \ref{Pro6SN} for $H$, and the fact that $\xi\geq \frac{4}{T^2}$,  we obtain
 	\begin{align*}
 	\mathcal{I}(s,\lambda;H)&\leq C\Big(\int_{\Omega_T}e^{-2s\alpha}|z |^2  d x\,dt
 	 + \int_{\Omega_T}e^{-2s\alpha}|a h |^2  d x\,dt
 	 \\
 	 &+\int_{\Omega_T}e^{-2s\alpha}|B\cdot\nabla h |^2  d x\,dt
 	+ \int_{\Gamma_T}e^{-2s\alpha}|b h_{\Gamma} |^2  d \sigma\,dt\\ &+\int_{\Gamma_T}e^{-2s\alpha}|B_{\Gamma}\cdot\nabla_{\Gamma} h_{\Gamma} |^2  d \sigma\,dt +  s^3\lambda^4\int_{\omega'\times(0,T)}\xi^3 e^{-2s\alpha}|h|^2dx\,dt\Big).\\
 	&
 	\leq C\Big(\int_{\Omega_T}e^{-2s\alpha}|z |^2  d x\,dt
 	+ \|a\|^2_{\infty}\int_{\Omega_T}\xi^3e^{-2s\alpha}| h |^2  d x\,dt
 	\\
 	&+ \|B\|^2_{\infty}\int_{\Omega_T}\xi e^{-2s\alpha}|\nabla h |^2  d x\,dt
 	+ \|b\|^2_{\infty}\int_{\Gamma_T}\xi^3e^{-2s\alpha}| h_{\Gamma} |^2  d \sigma\,dt\\ &+\|B_{\Gamma}\|^2_{\infty}\int_{\Gamma_T}\xi e^{-2s\alpha}|\nabla_{\Gamma} h_{\Gamma} |^2  d\sigma\,dt +  s^3\lambda^4\int_{\omega'\times(0,T)}\xi^3 e^{-2s\alpha}|h|^2dx\,dt\Big).
 	\end{align*}
 	By choosing $s_1$ large enough, one  has
 	\begin{equation}\label{In61SN}
 	\mathcal{I}(s,\lambda;H)\leq C\Big(\int_{\Omega_T}e^{-2s\alpha}|z |^2  d x\,dt
 	+  s^3\lambda^4\int_{\omega'\times(0,T)}\xi^3 e^{-2s\alpha}|h|^2dx\,dt\Big).\\
 	\end{equation}
 	From \eqref{In60SN} and \eqref{In61SN} and choosing  $s_1$ and $s_2$  large enough, we deduce
 	\begin{equation}\label{Syst.36NS}
 	\mathcal{I}(s,\lambda;Z) + \mathcal{I}(s,\lambda;H)\leq
 	\end{equation}
 	\begin{equation*}
 	C\left(
 	s^3\lambda^4\int_{\omega'\times(0,T)}\xi^3        e^{-2s\alpha}|z|^2dx\,dt+   s^3\lambda^4\int_{\omega'\times(0,T)}\xi^3 e^{-2s\alpha}|h|^2dx\,dt\right).
 	\end{equation*}
 	Thanks to \eqref{Assump10} and \eqref{Eq31D}, we have
 	\begin{equation}\label{Equ43SN}
 	h=-\partial_t z- \mathrm{div}(\mathcal{A}\nabla z)- \mathrm{div}(B z)+ a(x,t)z\quad  \text{in} \,\,\omega'_1.
 	\end{equation}
 	Using \eqref{Equ43SN} and the fact that $\text{supp}(\theta)\subset \omega'_1$, we find
 	{\small\begin{align*}\label{eq4.18}
 	&s^3\lambda^4\int_{\omega'\times(0,T)}\theta\xi^3 e^{-2s\alpha}|h|^2dx\,dt=\\
 	&s^3\lambda^4\int_{\omega'\times(0,T)}\theta \xi^3 e^{-2s\alpha}h (-\partial_t z- \mathrm{div}(\mathcal{A}\nabla z)- \mathrm{div}(B z)+ a(x,t)z)dx\,dt\\
 	&= s^3\lambda^4\int_{\omega'\times(0,T)}\xi^3 e^{-2s\alpha}z \partial_t h dx\,dt + s^3\lambda^4\int_{\omega'\times(0,T)} \partial_t (\theta\xi^3 e^{-2s\alpha}) z h dx\,dt\\
 	& \;+ s^3\lambda^4\int_{\omega'\times(0,T)} \mathrm{div}(\mathcal{A}\nabla( \theta\xi^3 e^{-2s\alpha}h))  z dx\,dt + s^3\lambda^4\int_{\omega'\times(0,T)} B\cdot\nabla(\theta\xi^3 e^{-2s\alpha}h)  z dx\,dt\\
 	&\; +s^3\lambda^4\int_{\omega'\times(0,T)}\theta \xi^3 e^{-2s\alpha}  a h  z dx dt.
 	\end{align*}}By the symmetry of the matrix $\mathcal{A}$, we obtain
 	\begin{align*}
 	\mathrm{div}(\mathcal{A}\nabla (\theta\xi^3 e^{-2s\alpha}h))=& \mathrm{div}(\mathcal{A}\nabla (\theta\xi^3 e^{-2s\alpha}))h+\mathrm{div}(\mathcal{A}\nabla h)\theta\xi^3 e^{-2s\alpha} \\
 	&+ 2\nabla(\theta\xi^3 e^{-2s\alpha})\cdot \mathcal{A}\nabla h,\\
 	\nabla(\theta\xi^3 e^{-2s\alpha}h)&=h\nabla(\theta\xi^3 e^{-2s\alpha})+\theta\xi^3 e^{-2s\alpha}\nabla h.
 	\end{align*}
 	Then
 	{\small\begin{align*}
 	&s^3\lambda^4\int_{\omega'\times(0,T)}\theta\xi^3 e^{-2s\alpha}|h|^2dx\,dt=\\ \nonumber
 	&  s^3\lambda^4\int_{\omega'\times(0,T)}\theta\xi^3 e^{-2s\alpha}z \partial_t h dx\,dt + s^3\lambda^4\int_{\omega'\times(0,T)}h z \partial_t (\xi^3 e^{-2s\alpha}) dx\,dt\nonumber\\
 	&  + s^3\lambda^4\int_{\omega'\times(0,T)} \mathrm{div}(\mathcal{A}\nabla( \theta\xi^3 e^{-2s\alpha})) h z dx\,dt  + s^3\lambda^4\int_{\omega'\times(0,T)}\mathrm{div}(\mathcal{A}\nabla h)\theta\xi^3 e^{-2s\alpha} z dx\,dt\nonumber\\
 	& +2 s^3\lambda^4\int_{\omega'\times(0,T)} \nabla(\theta\xi^3 e^{-2s\alpha})\cdot \mathcal{A}\nabla h  z dx\,dt +s^3\lambda^4\int_{\omega'\times(0,T)} B\cdot\nabla(\theta\xi^3 e^{-2s\alpha})  hz dx\,dt\nonumber\\
 	&+s^3\lambda^4\int_{\omega'\times(0,T)} \theta\xi^3 e^{-2s\alpha}B\cdot\nabla h\,z dx\,dt +s^3\lambda^4\int_{\omega'\times(0,T)}\xi^3 e^{-2s\alpha}  a h  z dx dt.
 	\end{align*}}Using the equation satisfied by $h=\alpha_1\psi^1 + \alpha_2\psi^2$, we find
 	{\small\begin{align}
 	&s^3\lambda^4\int_{\omega'\times(0,T)}\theta\xi^3 e^{-2s\alpha}|h|^2dx\,dt\leq\\\nonumber &\frac{\alpha_1}{\mu_1}s^3\lambda^4\int_{\omega'\times(0,T)}\theta\xi^3 e^{-2s\alpha}z^2 dx\,dt +\frac{\alpha_2}{\mu_2}s^3\lambda^4\int_{\omega'\times(0,T)}\theta\xi^3 e^{-2s\alpha}z^2 dx\,dt\\\nonumber
 	& + s^3\lambda^4\int_{\omega'\times(0,T)}hz \partial_t (\theta\xi^3 e^{-2s\alpha}) dx\,dt+ s^3\lambda^4\int_{\omega'\times(0,T)} \mathrm{div}(\mathcal{A}\nabla( \theta\xi^3 e^{-2s\alpha})) h z dx\,dt\\\nonumber
 	&+2 s^3\lambda^4\int_{\omega'\times(0,T)} \nabla(\theta\xi^3 e^{-2s\alpha})\cdot \mathcal{A}\nabla h  z dx\,dt +s^3\lambda^4\int_{\omega'\times(0,T)} B\cdot\nabla(\theta\xi^3 e^{-2s\alpha})  hz dx\,dt.\label{Es.036}
 	\end{align}}By the  properties of the functions $\alpha$ and $\xi$ given in Lemma \ref{Lem2.5},  we can estimate the the four last terms in the right hand side in the above inequality. To this end, fix $\epsilon$ small enough and using Young inequality, we deduce  for the first term
 	\begin{align}
 	&s^3\lambda^4\int_{\omega'\times(0,T)}\partial_t(\xi^3 e^{-2s\alpha}\theta) z h\,dx\,dt\leq C s^4\lambda^4\int_{\omega'\times(0,T)}\xi^5 e^{-2s\alpha} z h\,dx\, dt\\\nonumber
 	&\leq \frac{C}{\epsilon} s^7\lambda^8\int_{\omega'\times(0,T)} \xi^7 e^{-2s\alpha} \, |z|^2dx\,dt + C \epsilon s^3\lambda^4\int_{\omega'\times(0,T)} \xi^3 e^{-2s\alpha} \, |h|^2dx\,dt.\label{Es.037}
 	\end{align}
 	For the second one, we have
 	\begin{align}
 	&s^3\lambda^4\int_{\omega'\times(0,T)}\mathrm{div}(\mathcal{A}\nabla(\xi^3 e^{-2s\alpha}\theta)) h zdx\,dt\leq Cs^5\lambda^6\int_{\omega'\times(0,T)} e^{-2s\alpha} \,|h| |z|dx\,dt\\\nonumber
 	&\leq  \frac{C}{\epsilon} s^7\lambda^8\int_{\omega'\times(0,T)} \xi^7 e^{-2s\alpha} \, |z|^2dx\,dt + C \epsilon s^3\lambda^4\int_{\omega'\times(0,T)} \xi^3 e^{-2s\alpha} \,|h|^2dx\,dt. \label{Es.038}
 	\end{align}
 	The third one can be estimated as
 	\begin{align}
 	&-s^3\lambda^4\int_{\omega'\times(0,T)} z \nabla h\cdot\mathcal{A}\nabla(\xi^3 e^{-2s\alpha}\theta)dx\,dt\leq Cs^4\lambda^5\int_{\omega'\times(0,T)} \xi^4 e^{-2s\alpha} z |\nabla h|dx\,dt\\\nonumber	
 	&\leq \frac{C}{\epsilon} s^7\lambda^8\int_{\omega'\times(0,T)}\xi^7 e^{-2s\alpha} \, |z|^2dx\,dt
 	+ C \epsilon \lambda^{2}\int_{\omega'\times(0,T)} \xi e^{-2s\alpha} \,|\nabla h|^2dx\,dt.
 	\end{align}
 	For the  last one, we have
 	\begin{equation}\label{Es.039}
 	s^3\lambda^4\int_{\omega'\times(0,T)} B\cdot\nabla(\theta\xi^3 e^{-2s\alpha})  hz dx\,dt\leq s^4\lambda^5\int_{\omega'\times(0,T)}   hz dx\,dt
 	\end{equation}
 	\begin{equation*}
 	\leq \frac{C}{\epsilon}s^7\lambda^8\int_{\omega'\times(0,T)}\xi^7 e^{-2s\alpha} \, |z|^2dx\,dt
 	+ C \epsilon s^3\lambda^4\int_{\omega'\times(0,T)}\xi^3 e^{-2s\alpha} \, |h|^2dx\,dt.
 	\end{equation*}
 	Using \eqref{Syst.36NS}, \eqref{Es.036}-\eqref{Es.039} and choosing $\epsilon$ small enough, we find
 	\begin{align*}
 	&\mathcal{I}(s,\lambda; Z)+ \mathcal{I}(s,\lambda; H)  \leq \\ &\frac{\alpha_1}{\mu_1}s^3\lambda^4\int_{\omega'\times(0,T)}\xi^3 e^{-2s\alpha}z^2 dx\,dt +\frac{\alpha_2}{\mu_2}s^3\lambda^4\int_{\omega'\times(0,T)}\theta\xi^3 e^{-2s\alpha}z^2 dx\,dt\nonumber\\
 	&+ C s^7\lambda^8 \int_{\omega\times(0,T)} \xi^7 e^{-2s\alpha} |z|^2 \, dx\,dt.
 	\end{align*}
 	For $s_1$ and $s_2$ large enough, we conclude
 	\begin{align*}
 	\mathcal{I}(s,\lambda; Z)+ \mathcal{I}(s,\lambda; H)  \leq C s^7\lambda^8 \int_{\omega\times(0,T)} \xi^7 e^{-2s\alpha} |z|^2 \, dx\,dt.
 	\end{align*}
 	The proof is then finished.
 \end{proof}
 \begin{remark}
 	As mentioned in \cite{ArCaSa15}, the assumption \eqref{Assump10} is used to prove the Carleman estimate \eqref{Car4.13}.
 \end{remark}
 To prove the needed observability inequality, we are going to improve the Carleman  inequality \eqref{Car4.13}. To this end, following \cite{ArCaSa15},  we modify  the weight functions $\xi$ and $\alpha$. More precisely,
 we introduce the functions
 \begin{equation*}\label{eq:adjoint-system} l(t)=\left\{
 \begin{aligned}
 &T^2/4\quad &\text{if  } t\in(0,T/2),\\
 &	t(T-t) &\,\,\text{if}\,\, t\in (T/2,T),
 \end{aligned} \right.
 \end{equation*}
 and
 \begin{equation*}
 \overline{\xi}(x,t)=\frac{{e^{\lambda(2\|\eta^0\|_{\infty}+\eta^0(x) )}}}{l(t)},\quad \quad \overline{\alpha}(x,t)=\frac{e^{4\lambda\|\eta^0\|_{\infty} -e^{\lambda(2\|\eta^0\|_{\infty}+\eta^0(x) )}}}{l(t)}, \\
 \end{equation*}
 \begin{equation}
 \overline{\xi}^{\star}(t)=\min_{x\in\overline{\Omega}}\overline{\xi}(x,t)  \quad \text{and} \quad  	\overline{\alpha}^{\star}(t)=\max_{x\in\overline{\Omega}}\overline{\alpha}(x,t).
 \end{equation}
 Point out that the functions $\overline{\xi}$ and $\overline{\alpha}$ do not blow up at $t=0$, this will be important to get an improved Carleman estimate. We  also denote $$\mathcal{\overline{I}}(s,\lambda; Z)= \int_{\Omega\times(0,T)}e^{-2s\overline{\alpha}}\overline{\xi}^3|z|^2  dx\,dt  +\int_{\Gamma\times(0,T)}e^{-2s\overline{\alpha}}\overline{\xi}^3|z_{\Gamma}|^2  d\sigma\,dt,$$ where $s, \lambda>1$ and $Z=(z, z_{\Gamma})$ .
 With these definitions and the same notations of Theorem \ref{thmC4.3}, we have the following result.

 \begin{lemma}\label{lem4.5} Assume that \eqref{Assump10} holds, then there exist  $\lambda_4\geq 1$, $s_4\geq 1$ and $C>0$ such that  every solution $(Z, \Psi^1, \Psi^2)$  to \eqref{Sys.29NS}  satisfies
 	\begin{align*}
 	\|Z(0)\|^2_{\mathbb{L}^2} +\mathcal{\overline{I}}(s,\lambda; Z) + \mathcal{I}(s,\lambda; H)
 	\leq  C\int_{\omega\times(0,T)} \xi^7e^{-2s\alpha} |z|^2 \, dx\,dt
 	\end{align*}
 	for all $s\geq s_4$ and $\lambda\geq\lambda_4.$	
 \end{lemma}
 \begin{proof}
 	The proof  relies on the  Carleman estimate \eqref{Car4.13} and some energy estimates.
 	Following the strategy in \cite{ArCaSa15}, let us introduce a function $\eta\in C^1([0,T])$ such that
 	\begin{equation}
 	\eta=1\quad \text{in}\, [0,T/2],\quad \ \,\,\eta= 0\quad \text{in} \,\,[3T/4, T]\quad \text{and} \,\,\eta'\leq C/T^2.
 	\end{equation}
 	Set $P= \eta Z$. It is clear that $P=(p, p_{\Gamma})$ satisfies
 	{\footnotesize\begin{equation}\label{Eq5.16}
 	\left\{\begin{array}{ll}
 	{-\partial_t p-\mathrm{div}(\mathcal{A}\nabla p)- \mathrm{div}(pB) + a(x,t)p=\eta h 1_{\omega_d}+ \partial_t\,\eta z }& {\text { in } \Omega_{T},} \\
 	{-\partial_t p_{_{\Gamma}}   -\mathrm{div}_{\Gamma}(\mathcal{A}_{\Gamma}\nabla_{\Gamma} p_{\Gamma})- \mathrm{div}_{\Gamma}(p_{\Gamma}B_{\Gamma})+ p_{\Gamma} B\cdot \nu + \partial^{\mathcal{A}}_{\nu}p + b(x,t)p_{\Gamma}=\partial_t\,\eta z_{\Gamma}}& {\text { on } \Gamma_{T},} \\
 	{(p(T),p_{\Gamma}(T)) =0},& {\text { in } \Omega \times \Gamma}.
 	
 	\end{array}\right.
 	\end{equation}}From the energy estimate \eqref{Eq.19NS} for $P$, there exists a positive constant  $C$ such that
 	\begin{align*}
 	\|Z(0)\|^2_{\mathbb{L}^2}+ \|Z\|^2_{L^2(0,T/2; \mathbb{L}^2)}
 	\leq C\left(\frac{1}{T^2}\|Z\|^2_{L^2(T/2 , 3T/4; \mathbb{L}^2)} + \|H\|^2_{L^2(0,3T/4; \mathbb{L}^2))}\right).
 	\end{align*}
 	Using  the fact  that  the weight functions $e^{-2s\overline{\alpha}}$ and $\overline{\xi}$ (resp. $e^{-2s\alpha}$ and $\xi$ ) are bounded in $\overline{\Omega}\times[0, T/2]$ (resp. $\overline{\Omega}\times[T/2, 3T/4]$),  we deduce
 	\begin{align}
 	&\|Z(0)\|^2_{\mathbb{L}^2}+\int_{\Omega\times(0,T/2)}e^{-2s\overline{\alpha}}\overline{\xi}^3|z|^2  dx\,dt  +\int_{\Gamma\times(0,T/2)}e^{-2s\overline{\alpha}}\overline{\xi}^3|z_{\Gamma}|^2  d\sigma\,dt\nonumber\\
 	&\leq C\Big(\frac{1}{T^2}   \int_{\Omega\times(T/2,3T/4 )}  e^{-2s\alpha}\xi^3|z|^2 dx\,dt + \frac{1}{T^2}\int_{\Omega\times(T/2,3T/4 )}  e^{-2s\alpha}\xi^3|z_{\Gamma}|^2 d\sigma\,dt\nonumber\\
 	& \quad+\int_{\Omega\times(0,3T/4 )} |h|^2 dx\,dt +\int_{\Gamma\times(0,3T/4 )}  |h_{\Gamma}|^2 d\sigma\,dt\Big)\nonumber\\
 	& \leq C \big(\mathcal{I}(s,\lambda,Z) + \|H\|^2_{L^2(0,3T/4; \mathbb{L}^2)}\big).\label{Eq05.16}
 	\end{align}
 	On the other hand, since $\alpha=\overline{\alpha}$ and $\xi=\overline{\xi}$ in $(T/2, T)$, we have
 	\begin{equation}\label{Eq5.18}
 	\int_{\Omega\times(T/2, T)}e^{-2s\overline{\alpha}}\overline{\xi}^3|z|^2  dx\,dt  +\int_{\Gamma\times(T/2, T)}e^{-2s\overline{\alpha}}\overline{\xi}^3|z_{\Gamma}|^2  dx\,dt
 	\leq C \mathcal{I}(s,\lambda, Z).
 	\end{equation}
 	From \eqref{Eq05.16} and  \eqref{Eq5.18}, we obtain
 	\begin{align*}
 	&\|Z(0)\|^2_{\mathbb{L}^2} +\int_{\Omega\times(0,T)}e^{-2s\overline{\alpha}}\overline{\xi}^3|z|^2  dx\,dt  +\int_{\Gamma\times(0,T)}e^{-2s\overline{\alpha}}\overline{\xi}^3|z_{\Gamma}|^2  dx\,dt
 	\leq \\
 	&C\Big( \mathcal{I}(s,\lambda, Z)+ \|H\|^2_{L^2(0,3T/2; \mathbb{L}^2)} \Big).
 	\end{align*}
 	Using again the energy estimate  and the fact that the  weight functions $e^{-2s\overline{\alpha}}$ and $\overline{\xi}$ are bounded in $[0, 3T/4],$ we see that
 	{\small$$ \|H\|^2_{L^2(0,3T/2; \mathbb{L}^2)}\leq C(\frac{\alpha_1}{\mu_1}+ \frac{\alpha_2}{\mu_2})\left(\int_{\Omega\times(0,T)}e^{-2s\overline{\alpha}}\overline{\xi}^3|z|^2  dx\,dt  +\int_{\Gamma\times(0,T)}e^{-2s\overline{\alpha}}\overline{\xi}^3|z_{\Gamma}|^2  d\sigma\,dt\right).$$}For $\mu_1$ and  $\mu_2$ large enough,  we deduce
 	\begin{align*}
 	\|Z(0)\|^2_{\mathbb{L}^2} +\int_{\Omega\times(0,T)}e^{-2s\overline{\alpha}}\overline{\xi}^3|z|^2  dx\,dt  +\int_{\Gamma\times(0,T)}e^{-2s\overline{\alpha}}\overline{\xi}^3|z_{\Gamma}|^2  d\sigma\,dt
 	\leq C  \mathcal{I}(s,\lambda, Z).
 	\end{align*}
 	Then, we conclude that
 	\begin{align*}
 	\|Z(0)\|^2_{\mathbb{L}^2} + \mathcal{\overline{I}}(s,\lambda; Z)+ \mathcal{I}(s,\lambda;H)
 	\leq  C(\big(\mathcal{I}(s,\lambda;Z)+ \mathcal{I}(s,\lambda;H)\big).
 	\end{align*}
 	By Carleman estimate \eqref{Car4.13} and this last inequality, we find
 	\begin{align*}
 	\|Z(0)\|^2_{\mathbb{L}^2} +\mathcal{\overline{I}}(s,\lambda; Z) +\mathcal{I}(s,\lambda;H)
 	\leq  C\int_{\omega\times(0,T)} e^{-2s\alpha}\xi^7 |z|^2 \, dx\,dt.
 	\end{align*}
 	This finishes the proof.
 \end{proof}

 \subsection{Observability and null controllability} In this section we will prove our main result of controllability given in Theorem \ref{th4.1SN}. Let us first recall  that the characterization \eqref{chara.1}  of the follower controls adds two additional equations to the  system \eqref{Sys1.1SN}, and our problem is then reduced to look for a control $f\in L^2(0,T; L^2(\omega))$ such that the solution $(Y, \Phi^1, \Phi^2)$ to the optimality system \eqref{Sys.28NS} satisfies $Y(T)=0$. To this end, we  shall prove the following observability inequality.
 \begin{proposition}\label{Pro4.2}
 	Assume that \eqref{Assump10} holds   and $\mu_i>0$, $i=1,2$,  are large enough.
 	Then, there exist a constant $C>0$ and a positive weight function $\rho=\rho(t)$ blowing up at $t=T$, such that, for any $Z_T\in \mathbb{L}^2$, the solution $(Z, \Psi^1, \Psi^2)$ to \eqref{Sys.29NS} satisfies the following inequality
 	\begin{align}\label{obin4.3}
 	&	\|Z(0)\|^2_{\mathbb{L}^2} +
 	\sum_{i=1}^{2} \int_{\Omega_T} \rho^{-2} |\psi^i|^2 dx \,dt + \sum_{i=1}^{2} \int_{\Gamma_T} \rho^{-2} |\psi_{\Gamma}^i|^2 d\sigma \,dt\leq C\int_{\omega\times(0,T)}  |z|^2 \, dx\,dt.
 	\end{align}
 \end{proposition}
 \begin{proof}
 	Fix $s$ large enough and set $\rho=e^{s\alpha^{\star}}$. Thus $\rho$
 	is a  positive function blowing up at $t = T$.  For $i=1,2$, using the equation  satisfied by  $\Psi^i$ in \eqref{Sys.28NS}, we readily see that
 	{\small\begin{align*}
 	&\frac{1}{2}\frac{d}{dt}(\|\rho^{-1}\Psi^i(t)\|_{\mathbb{L}^2}^2) +   \int_{\Omega}\rho^{-2}\mathcal{A}\nabla\psi^i\cdot\nabla \psi^i dx \, -  \int_{\Omega}\rho^{-2}\psi^i B \cdot\nabla \psi^i dx+ \int_{\Omega} a |\rho^{-1}\psi^i|^2 dx \\
 	& + \int_{\Gamma}\rho^{-2}\mathcal{A}_{\Gamma}\nabla_\Gamma \psi^i_\Gamma\cdot \nabla_\Gamma \psi^i_\Gamma d\sigma \,-\int_{\Gamma}\rho^{-2}\psi^i B_{\Gamma} \cdot\nabla_{\Gamma} \psi^i_{\Gamma} d\sigma  + \int_{\Gamma} b |\rho^{-1}\psi^i_\Gamma|^2 d\sigma = \\
 	&\frac{1}{\mu_i}\|\rho^{-2}z \psi^i\|^2_{L^2(0,T; L^2(\omega_i))}  + \rho_t\rho^{-3} \|\Psi^i\|^2_{\mathbb{L}^2}.
 	\end{align*}}By Young inequality and for a positive $\lambda$, we have
 	\begin{equation*}
 	\int_{\Omega}\rho^{-2} \psi^i B \cdot\nabla \psi^i dx\leq \frac{\lambda}{2}\|B\|^2_{\infty}\|  \rho^{-1} \psi^i\|^2_{L^2(\Omega)}  + \frac{1}{2\lambda}\| \rho^{-1}\nabla\psi^i\|^2_{L^2(\Omega)},
 	\end{equation*}
 	\begin{equation*}
 	\int_{\Gamma}\rho^{-2} \psi^i B_{\Gamma} \cdot\nabla_{\Gamma} \psi^i_{\Gamma} dx\leq \frac{\lambda}{2}\|B_{\Gamma}\|^2_{\infty}\|  \rho^{-1} \psi^i_{\Gamma}\|^2_{L^2(\Gamma)}  + \frac{1}{2\lambda}\| \rho^{-1}\nabla_{\Gamma}\psi^i_{\Gamma}\|^2_{L^2(\Gamma)} \quad \text{in} \,\,(0,T).
 	\end{equation*}	
 	Using the ellipticity of $\mathcal{A}$ and $\mathcal{A}_{\Gamma} $ and choosing $\lambda$ large enough, we obtain
 	\begin{align*}
 	\frac{1}{2}\frac{d}{dt}(\|\rho^{-1}\Psi^i(t)\|_{\mathbb{L}^2}^2) &\leq C(1+ \|a\|_{\infty}+\|b\|_{\infty}+\|B\|^2_{\infty}+\|B_{\Gamma}\|^2_{\infty})\|\Psi^i(t)\|_{\mathbb{L}^2}^2 \\
 	&+ \frac{1}{\mu_i}\|\rho^{-2}z\|^2_{L^2(0,T; L^2(\omega_i))},\quad t\in(0,T).
 	\end{align*}
 	Thus, from Gronwall 's lemma and the fact that $\Psi^i(0) = 0$, it follows that
 	\begin{align}\label{eq4.22}
 	\|\rho^{-2}(t)\Psi^i(t)\|^2_{\mathbb{L}^2} &\leq  C \|\rho^{-2}z \|^2_{L^2(0,T; L^2(\omega_i))},\quad  t\in (0,T).
 	\end{align}
 	Using Lemma \ref{lem4.5}, we find
 	\begin{equation}\label{Eq062}
 	\|Z(0)\|^2_{\mathbb{L}^2}
 	+\int_{\Omega_T}\rho^{-2}|z|^2  dx\,dt + \mathcal{I}(s,\lambda; H)\leq C\int_{\omega_T} e^{-2s \alpha}\xi^7 |z|^2 \, dx\,dt.
 	\end{equation}
 	From the inequalities  \eqref{eq4.22} and \eqref{Eq062}  and the fact that the function $e^{-2s \alpha}\xi^7$ is bounded in $\Omega_T$,
 	we obtain
 	\begin{align*}
 	&	\|Z(0)\|^2_{\mathbb{L}^2}
 	+\int_{\Omega_T}\rho^{-2}|\psi^1|^2  dx\,dt\, + \int_{\Gamma_T}\rho^{-2}|\psi^1_{\Gamma}|^2  d\sigma\,dt +\nonumber \\	&+\int_{\Omega_T}\rho^{-2}|\psi^2|^2  dx\,dt + \int_{\Gamma_T}\rho^{-2}|\psi^2_{\Gamma}|^2  d\sigma\,dt\leq C\int_{\omega_T}  |z|^2 \, dx\,dt.
 	\end{align*}
 	This ends the proof.
 \end{proof}
 Now, we prove the controllability result in Theorem \ref{th4.1SN}. Since we have proved the existence and uniqueness of Nash equilibrium in Proposition \ref{Prop04}, it remains to  prove the following result.
 \begin{proposition}\label{Lm5.5}
 	Let  $\rho=\rho(t)$ the  weight function given in Proposition \ref{Pro4.2}. Then, for any  $y_{i,d} \in \mathcal{H}_{i,d}$, $i=1,2,$ satisfying \eqref{Ass11SN} and  $Y_{0} \in \mathbb{L}^2$,  there exists a control $f \in L^{2}(0,T; L^2(\omega))$ with minimal norm such that
 	\begin{equation}\label{EsCt47}
 	\|f\|_{L^{2}(0,T; L^2(\omega))} \leq C\left(\left\|Y_{0}\right\|_{\mathbb{L}^2}+\left\|\rho y_{1,d}\right\|_{L^{2}(\Omega_T)} + \left\|\rho y_{2,d}\right\|_{L^{2}(\Omega_T)}\right),
 	\end{equation}
 	and the associated state satisfies $Y(T)=0,$ where $(Y, \Phi^1,\Phi^2)$ is the solution to \eqref{Sys.28NS} and  $C$	is a positive constant.
 	
 \end{proposition}
 \begin{proof}
 	
 	Multiplying the solution $(Y, \Phi^1, \Phi^2)$ to \eqref{Sys.28NS} by the solution $(Z, \Psi^1,\Psi^2)$ to \eqref{Sys.29NS} and integrating by parts, we find
 	\begin{equation}\label{Eq48NS}
 	\langle Y(T), Z(T)\rangle_{\mathbb{L}^2}-  \langle Y(0), Z(0)\rangle_{\mathbb{L}^2}=  \int_{\omega\times (0,T)} f\psi\, dx \,dt - \sum_{i=1}^{2}\alpha_i \int_{\omega_{i,d}\times (0,T)} y_{i,d} \psi^i dx \,dt.
 	\end{equation}
 	Thus,  the null controllability property is equivalent to find, for each $Y_0\in\mathbb{L}^2$, a control $f$ such that,
 	for any $Z_T\in\mathbb{L}^2$, one has
 	\begin{equation}
 	\int_{\omega\times (0,T)} f\psi\, dx \,dt= -\langle Y(0), Z(0)\rangle_{\mathbb{L}^2} +\sum_{i=1}^{2}\alpha_i \int_{\omega_{i,d}\times (0,T)} y_{i,d} \psi^i dx \,dt.
 	\end{equation}
 	To this end, let $\epsilon>0$ and  $Z_T\in \mathbb{L}^2.$ Introduce the following functional
 	\begin{align*}
 	J_{\epsilon}(Z_T)&= \frac{1}{2}\int_{\omega\times(0,T)}|z|^2 dx\,dt +\epsilon\|Z_T\|_{\mathbb{L}^2}   +\langle Y_0 , Z(0)\rangle_{\mathbb{L}^2}\\
 	&- \sum_{i=1}^{2}\alpha_i \int_{\omega_{i,d}\times(0,T)} y_{i,d} \psi^i dx \,dt.
 	\end{align*}
 	It is  clear that $J_{\epsilon}: \mathbb{L}^2\longrightarrow\mathbb{R}$ is continuous and convex. Moreover, from Young inequality together with the observability inequality \eqref{obin4.3}, we have, for $\delta>0$,
 	{\small\begin{align*}
 		\langle Y(0) , Z(0)\rangle_{\mathbb{L}^2}&\geq -\frac{1}{2\delta}\|Y_0\|^2_{\mathbb{L}^2} - \frac{\delta}{2}\|Z(0)\|^2_{\mathbb{L}^2}\\
 		&\geq  - \frac{\delta}{2}\left( C\int_{\omega\times(0,T)}   |z|^2 dx dt   +\int_{\Omega_T}\rho^{-2}|\psi^1|^2  dx\,dt + \int_{\Omega_T}\rho^{-2}|\psi^2|^2  dx\,dt\right)\\
 		&-\frac{1}{2\delta }\|Y_0\|^2_{\mathbb{L}^2}
 		\end{align*}}and
 	\begin{align*}
 	-\sum_{i=1}^{2}\alpha_i \int_{\omega_{i,d}\times(0,T)} y_{i,d} \psi^i dx \,dt&\geq -\frac{1}{2\delta}\sum_{i=1}^{2}\ \int_{\omega_{i,d}\times(0,T)} \alpha_i^2 |y_{i,d}|^2 dx \,dt\\
 	&-  \frac{\delta}{2}\sum_{i=1}^{2} \int_{\omega_{i,d}\times(0,T)} \rho^{-2} |\psi^i|^2 dx \,dt.
 	\end{align*}
 	Choosing $\delta=\frac{1}{2C}$ and using the above inequalities,  we get
 	{\small\begin{equation}
 	J_{\epsilon}(Z_T)\geq \frac{1}{4}\int_{\omega\times(0,T)}|z|^2 dx\,dt +\epsilon\|Z_T\|_{\mathbb{L}^2}   - C\left( \|Y_0\|^2_{\mathbb{L}^2}
 	+ \sum_{i=1}^{2}\alpha^2_i \int_{\omega_{i,d}\times(0,T)} \rho^2|y_{i,d}|^2 dx \,dt\right).
 	\end{equation}}Consequently, $J_{\epsilon}$ is  coercive in $\mathbb{L}^2$, and  then $J_{\epsilon}$  admits a unique minimizer $Z^{\epsilon}_T$. If $Z^{\epsilon}_T \neq0$, we  have
 	\begin{equation}
 	\langle J'(Z^{\epsilon}_T), Z_T\rangle_{\mathbb{L}^2} =0,\quad \,\,Z_T\in\mathbb{L}^2.
 	\end{equation}
 	Then, for all $Z_T\in \mathbb{L}^2$, we have
 	{\small\begin{equation}\label{Eq51NS}
 	\int_{\omega\times(0,T)}z^{\epsilon}\,z dx\,dt +\epsilon\langle\frac{Z^{\epsilon}_T}{\|Z^{\epsilon}_T\|_{\mathbb{L}^2}}, Z_T\rangle_{\mathbb{L}^2}    +\langle Y_0, Z(0)\rangle_{\mathbb{L}^2}- \sum_{i=1}^{2}\alpha_i \int_{{\omega_{i,d}\times(0,T)}} y_{i,d} \psi^i dx \,dt=0,
 	\end{equation}}where we have denoted by $(Z_{\epsilon}, \Psi^1_{\epsilon},\Psi^2_{\epsilon})$ the solution to \eqref{Sys.28NS} with $Z_T=Z^{\epsilon}_T$.
 	Take $f_{\epsilon}=z_{\epsilon}$ in \eqref{Eq48NS}, we find
 	\begin{equation}
 	\epsilon\langle\frac{Z^{\epsilon}_T}{\|Z^{\epsilon}_T\|_{\mathbb{L}^2}}, Z_T\rangle_{\mathbb{L}^2}    +\langle Y_{\epsilon}(T), Z_T\rangle_{\mathbb{L}^2}=0,\quad  Z_T\in \mathbb{L}^2.
 	\end{equation}
 	Hence
 	\begin{equation}\label{eq53NS}
 	\|Y_{\epsilon}(T)\|_{\mathbb{L}^2}\leq \epsilon.
 	\end{equation}
 	Taking $z=z_{\epsilon}$ in \eqref{Eq51NS}, and using observability inequality \eqref{obin4.3} together with  Young inequality, we deduce
 	\begin{equation}\label{eq54NS}
 	\|f_{\epsilon}\|^2_{L^2(0,T; L^2(\omega))}\leq C\left(\|Y_0\|^2+ \sum_{i=1}^{2}\alpha^2_i \int_{\omega_{i,d}\times(0,T)} \rho^2 y^2_{i,d} dx \,dt\right).
 	\end{equation}
 	If $Z^{\epsilon}_T=0$, arguing
 	as in \cite{FPZ91}, we deduce that
 	\begin{equation}\label{Ezu.1}
 	\lim\limits_{t\rightarrow 0^{+}}\frac{J_{\epsilon}(t \,Z_T)}{t}\geq0, \quad  Z_T \in \mathbb{L}^2.
 	\end{equation}
 	Using \eqref{Ezu.1} and take $f_{\epsilon} = 0$, we obtain \eqref{eq53NS} and \eqref{eq54NS}. By \eqref{eq54NS}, we deduce that there exist $f\in L^2(0,T; L^2(\omega))$ and a subsequence, still denoted by $f_{\epsilon}$, such that
 	$$ f_{\epsilon}\longrightarrow  f\quad  \text{weakly in} \,\, L^2(0,T; L^2(\omega)).$$
 	By energy estimate, we deduce
 	\begin{equation}\label{Eq56}
 	Y_{\epsilon}(T)\longrightarrow  Y(T)\quad  \text{weakly in} \,\,  \mathbb{L}^2.
 	\end{equation}
 	Using \eqref{eq53NS} and \eqref{Eq56}, we deduce that $Y(T)=0.$ This concludes the null controllability result.
 \end{proof}
As a consequence of the above HUM method, see for instance \cite{Pu08}, the control we have constructed is one of the minimal norm, and  it is characterized as follows.
 \begin{corollary}
  Let  $Z^{\epsilon}_T$ be the unique minimizer of $J_{\epsilon}$, then the leader control with minimal norm is given by the limit of $\hat{f}_{\epsilon}=z1_{\omega}=z_{\epsilon} 1_{\omega}$ as $\epsilon$ goes to zero, where $(Y, Z,\Phi^1,\Phi^2, \Psi^1, \Psi^2 )= \Big((y,y_{\Gamma}), (z,z_{\Gamma}), (\varphi^1, \varphi^1_{\Gamma}),(\varphi^2, \varphi^2_{\Gamma}),(\psi^1, \psi^1_{\Gamma}),(\psi^2, \psi^2_{\Gamma}) \Big)$ satisfies
{\small \begin{equation}
 \left\{\begin{array}{ll}
 {\partial_t y-\dv(\mathcal{A}\nabla y)  +B(x,t)\cdot \nabla y+ a(x,t)y=1_{\omega}z -\frac{1}{\mu_1} \varphi^1 1_{\omega_1}-\frac{1}{\mu_2} \varphi^2 1_{\omega_2} },& {\mathrm{in} \,\Omega_{T}} \\
 {	\partial_t y_{_{\Gamma}}   -\dv_{\Gamma}(\mathcal{A}_{\Gamma}\nabla_{\Gamma} y_{\Gamma})  +B_{\Gamma}(x,t)\cdot \nabla_{\Gamma} y_{\Gamma} + \partial^{\mathcal{A}}_{\nu}y + b(x,t)y_{\Gamma}=0},& {\mathrm{on} \,\Gamma_{T}} \\
 {-\partial_t z- \dv(\mathcal{A}\nabla z)- \dv(B z)+ a(x,t)z=\alpha_1 \psi^1 1_{\omega_{1, d}}+\alpha_2 \psi^2 1_{\omega_{1,d}} },& {\mathrm{in}\, \Omega_{T}} \\
 {	-\partial_t z_{_{\Gamma}}   -\dv_{\Gamma}(\mathcal{A}_{\Gamma}\nabla_{\Gamma} z_{\Gamma})- \dv_{\Gamma}(B_{\Gamma} z_{\Gamma})+ z_{\Gamma}B\cdot\nu   + \partial^{\mathcal{A}}_{\nu}z+ b(x,t)z_{\Gamma}=0},& {\mathrm{on}\, \Gamma_{T}} \\
 {	-\partial_t \varphi^i-\dv(\mathcal{A}\nabla \varphi^i)  - \dv(B \varphi^i) + a(x,t)\varphi^i=\alpha_i(y-y_{i,d})1_{\omega_{i,d}}},& {\mathrm{in} \,\Omega_{T}}\\
 {	-\partial_t \varphi^i_{_{\Gamma}}   - \dv_{\Gamma}(\mathcal{A}_{\Gamma}\nabla_{\Gamma} \varphi^i_{\Gamma})-\dv_{\Gamma}(B_{\Gamma} \varphi^i_{\Gamma})+ \varphi^i_{\Gamma} B\cdot\nu  + \partial^{\mathcal{A}}_{\nu}\varphi^i + b(x,t)\varphi^i_{\Gamma}=0}, & {\mathrm{on} \,\Gamma_{T}}\\
 {	\partial_t \psi^i- \dv(\mathcal{A}\nabla \psi^i)+ B\cdot\nabla\psi^i + a(x,t)\psi^i=-\frac{1}{\mu_i}z\,1_{\omega_{i}}}, & {\mathrm{in}\, \Omega_{T}}\\
 {	\partial_t \psi^i_{_{\Gamma}}   -\dv_{\Gamma}(\mathcal{A}_{\Gamma}\nabla_{\Gamma} \psi^i_{\Gamma})+ B_{\Gamma}\cdot\nabla\psi^i_{\Gamma}  + \partial^{\mathcal{A}}_{\nu}\psi^i + b(x,t)\psi^i_{\Gamma}=0}, & {\mathrm{on}\, \Gamma_{T}}\\
 {(y(0),y_{\Gamma}(0))=Y_0},& {\mathrm{in} \,\Omega \times \Gamma} \\
 {(z(T),z_{\Gamma}(T))=Z^{\epsilon}_T},& {\mathrm{in} \,\Omega \times \Gamma} \\
 {(\varphi^i(T),\varphi^i_{\Gamma}(T)) =0},& {\mathrm{in}\, \Omega \times \Gamma}\\
 {(\psi^i(0),\psi^i_{\Gamma}(0)) =0},& {\mathrm{in} \, \Omega \times \Gamma}\\
 i=1,2.
 \end{array}\right.
 \end{equation}}\end{corollary}
 \begin{remark}
 	(a) For a penalized HUM method of   constructing controls with minimal norm, we refer to \cite{RF13} and \cite{GL08}.\\
 	(b) We shall give the idea behind the form of the observability inequality \eqref{obin4.3} and we  prove the equivalence between the controlability of \eqref{Sys.28NS} and the observability inequality \eqref{obin4.3}. To this end, let us introduce the following functionals
 	\begin{equation*}
 	\mathcal{L}_T: L^2(0,T; L^2(\omega))\longrightarrow \mathbb{L}^2
 	\end{equation*}
 	and
 	\begin{equation*}
 	\mathcal{R}_T:  \mathbb{L}^2\times L^2_{\rho}(0,T; L^2(\omega_{1,d}))\times L^2_{\rho}(0,T; L^2(\omega_{2,d}))\longrightarrow \mathbb{L}^2
 	\end{equation*}
 	defined by
 	\begin{equation*}
 	\mathcal{L}_T(f)= Y(T,Y_0=0, f, y_{1,d}=0,y_{2,d}=0) = Y_1(T)
 	\end{equation*}
 	and
 	\begin{equation*} \mathcal{R}_T(Y_0, y_{1,d},y_{2,d})= Y(T,Y_0, f=0, y_{1,d},y_{2,d})=Y_2(T),
 	\end{equation*}
 	where   $$L^2_{\rho}(0,T; L^2(\omega_{i,d}))=\left\{u\in L^2(0,T; L^2(\omega_{i,d})): \int_{\omega_{i,d}\times(0,T)}\rho^2 u^2 dx dt< \infty\right\}$$
 	and  $(Y, \Phi^1,\Phi^2)$ is the solution to \eqref{Sys.28NS}.
 	Multiplying the adjoint system  \eqref{Sys.29NS} by $(Y, \Phi^1,\Phi^2)$ and integrating by parts, we get
 	\begin{align}
 	&\langle Y(T), Z(T)\rangle_{\mathbb{L}^2} -\langle Y_0,  Z(0)\rangle_{\mathbb{L}^2} +	\sum\limits_{i=1}^{2}\int_{\omega_{i, d}\times(0,T)}  \alpha_{i}y \psi^{i} dx dt=\nonumber\\
 	&\int_{\omega\times(0,T)}  z f dx\,dt - \sum\limits_{i=1}^{2}\frac{1}{\mu_{i}}\int_{\omega_{i}\times(0,T)} z\varphi^{i} dx dt \label{Eq58}.
 	\end{align}
 	In the same way, multiplying the  system  \eqref{Sys.28NS} by $(Z, \Psi^1, \Psi^2)$ and integrating by parts, we get
 	\begin{equation} \label{Eq59}
 	\int_{\omega_{i,d}\times(0,T)} \alpha_{i}\psi^i(y-y_{i, d})\,dx\,dt= -\frac{1}{\mu_{i}}\int_{\omega_{i}\times(0,T)} z\varphi^i \,dx\,dt  .
 	\end{equation}
 	If $Y_0=0$ and  $y_{1,d}=y_{1,d}=0$, by using \eqref{Eq58} and \eqref{Eq59}, we find
 	$$
 	\left\langle Y_1(T), Z_{T}\right\rangle=\int_{\omega\times(0,T)} z f\,dx\,dt.
 	$$
 	This proves that
 	\begin{equation*}
 	\mathcal{L}^{\star}_T(Z_T)= z 1_{\omega}.
 	\end{equation*}
 	If now $f=0$, using again \eqref{Eq58} and \eqref{Eq59}, we get
 	\begin{equation*}
 	\langle Y_2(T), Z(T)\rangle_{\mathbb{L}^{2}}=\langle Y_0, Z(0)\rangle_{\mathbb{L}^{2}}-\sum_{i=1}^{2} \alpha_{i} \int_{\omega_{i,d}\times (0,T)} y_{i, d} \psi^{i} d x d t.
 	\end{equation*}
 	That is
 	\begin{equation*}
 	\mathcal{R}_{T}^{\star}\left(Z_{T}\right)=\left(Z(0), -\alpha_{1} \rho^{-1} \psi^{1}, -\alpha_{2} \rho^{-1} \psi^{2}\right).
 	\end{equation*}
 	The relation between observability inequality \eqref{obin4.3}  and null controlability of the system \eqref{Sys.28NS}  follows from  Theorem 1.18  of \cite{WWZ18}.
 \end{remark}
 \section{Similar results for semilinear problems}\label{sectionsemli}
 In this section, following \cite{ArCaSa15, BCMO20}, and using  the results obtained in the linear case with a fixed point argument, we deduce similar results for the following  semilinear problem
 \begin{equation}
 \left\{\begin{array}{ll}
 {\partial_t y-\text{div}(\mathcal{A}\nabla y)  + F (y,\nabla y)= f1_{\omega}+ v_1 1_{\omega_1}+ v_2 1_{\omega_2} }& {\text { in } \Omega_T,} \\
 {\partial_t y_{_{\Gamma}}   - \text{div}_{\Gamma}(\mathcal{A}_{\Gamma}\nabla_{\Gamma} y_{\Gamma})    + \partial_{\nu}^{\mathcal{A}}y+ G (y_{\Gamma},\nabla_{\Gamma} y_{\Gamma})=0} &\,\,{\text {on} \,\Gamma_T,} \\
 {(y(0),y_{\Gamma}(0)) =(y_0,y_{\Gamma,0})} & {\text { in } \Omega \times \Gamma}, \label{nl2.62_}
 \end{array}\right.
 \end{equation}
 with  $F: (s, \zeta_1,...,\zeta_N )\longmapsto F(s, \zeta_1,...,\zeta_N )$  and    $G: (s, \zeta_1,...,\zeta_N )\longmapsto G(s, \zeta_1,...,\zeta_N )$ are in $C^1(\mathbb{R}\times \mathbb{R}^N)$ such that $F(0)=G(0)=0$. We assume furthermore that there exist $L_F , L_G>0$ such that
 \begin{align}
 |F(s,\zeta)-F(r,\zeta')|&\leq L_F(|s-r| + \|\zeta- \zeta'\|),\label{Lip2.63}\\
 |G(s,\zeta)-G(r,\zeta')|&\leq L_G(|s-r| + \|\zeta- \zeta'\|) \label{Lip2.64}
 \end{align}
 for all \,\,$s, r \in \mathbb{R}$\, and \,\,$\zeta, \zeta'\in \mathbb{R}^N$. \\
 Since the method is standard, we only give the main ideas.
 In the semilinear framework, the corresponding functionals  $J_1$ and  $J_2$ are not convex in general.
 For this reason, we must consider the following  weaker Nash equilibrium.
 \begin{definition} Let  $f\in L^2(0,T; L^2(\omega))$ be given. The pair $(v_1, v_2)$ is called Nash
 	quasi-equilibrium of $(J_1, J_2)$ if
 	\begin{equation*}
 	\left\{\begin{array}{ll}
 	{\frac{\partial J_1}{\partial v_1}(f; v^{\star}_1,v^{\star}_2)(v)=0 \quad \forall \, v\in  L^2(0,T;L^2( \omega_1)),} \\
 	{\frac{\partial J_2}{\partial v_2}(f; v^{\star}_1,v^{\star}_2)(w)=0 \quad \forall \, w\in  L^2(0,T;L^2( \omega_2)).}
 	\end{array}\right.
 	\end{equation*}
 \end{definition}
Introduce the following notations
\begin{align*}
&\tilde{ F}_1(y)=\partial_s F(y, \nabla y),\,\, \tilde{ F}_2(y)=(\partial_{\zeta_1}G(y, \nabla y),..., \partial_{\zeta_N}F(y, \nabla y)),\\
&\tilde{ G}_1(y)=\partial_s G(y, \nabla y),\,\, \tilde{ G}_2(y)=(\partial_{\zeta_1}G(y, \nabla y),..., \partial_{\zeta_N}F(y, \nabla y)).
\end{align*}
  Using the same ideas as in \cite{ArCaSa15}, the Nash quasi-equilibrium $(v^{\star}_1, v^{\star}_2)$ of $(J_1, J_2)$ is given by
\begin{equation}\label{chara.1_}
v^{\star}_i=-\frac{1}{\mu_i} \varphi^i|_{\omega_i\times(0,T)} \quad \text{for} \,\,i=1,2,
\end{equation}
where $(\varphi^1, \varphi^2)$ satisfies
{\tiny\begin{equation*}
	\left\{\begin{array}{ll}
	{-\partial_t \varphi^i-\mathrm{div}(\mathcal{A}\nabla \varphi^i)- \mathrm{div}(\varphi^i \tilde{ F}_2(y)) + \tilde{ F}_1(y)\varphi^i=\alpha_i(y-y_{i,d})1_{\omega_{i,d}} }& {\text { in } \Omega_{T},} \\
	{-\partial_t \varphi^i_{_{\Gamma}}   -\mathrm{div}_{\Gamma}(\mathcal{A}_{\Gamma}\nabla_{\Gamma} \varphi_{\Gamma}^i)- \mathrm{div}_{\Gamma}(\varphi_{\Gamma}^i\tilde{G}_2(y_{\Gamma}))+ \varphi^i_{\Gamma} \tilde{ F}_2(y)\cdot \nu + \partial^{\mathcal{A}}_{\nu}\varphi^i + \tilde{ G}_1(y_{\Gamma})\varphi^i_{\Gamma}=0}& {\text { on } \Gamma_{T},} \\
	{(\varphi^i(T),\varphi^i_{\Gamma}(T)) =0}& {\text { in } \Omega \times \Gamma,}\\
	i=1,2,
	\end{array}\right.
	\end{equation*}}Accordingly, we have the next optimality system
{\tiny\begin{equation}\label{Eq25S0_}
	\left\{\begin{array}{ll}
	{\partial_t y-\text{div}(\mathcal{A}\nabla y)  + F (y,\nabla y)= f1_{\omega}+ v_1 1_{\omega_1}+ v_2 1_{\omega_2} }& {\text { in } \Omega_T,} \\
	{\partial_t y_{_{\Gamma}}   - \text{div}_{\Gamma}(\mathcal{A}_{\Gamma}\nabla_{\Gamma} y_{\Gamma})    + \partial_{\nu}^{\mathcal{A}}y+ G (y_{\Gamma},\nabla_{\Gamma} y_{\Gamma})=0} &\,\,{\text {on} \,\Gamma_T,} \\
	{-\partial_t \varphi^i-\mathrm{div}(\mathcal{A}\nabla \varphi^i)- \mathrm{div}(\varphi^i \tilde{ F}_2(y)) + \tilde{ F}_1(y)\varphi^i=\alpha_i(y-y_{i,d})1_{\omega_{i,d}} }& {\text { in } \Omega_{T},} \\
	{-\partial_t \varphi^i_{_{\Gamma}}   -\mathrm{div}_{\Gamma}(\mathcal{A}_{\Gamma}\nabla_{\Gamma} \varphi_{\Gamma}^i)- \mathrm{div}_{\Gamma}(\varphi_{\Gamma}^i\tilde{G}_2(y_{\Gamma}))+ \varphi^i_{\Gamma} \tilde{ F}_2(y)\cdot \nu + \partial^{\mathcal{A}}_{\nu}\varphi^i + \tilde{ G}_1(y_{\Gamma})\varphi^i_{\Gamma}=0}& {\text { on } \Gamma_{T},} \\
	{(y(0),y_{\Gamma}(0)) =(y_0,y_{\Gamma,0})} & {\text { in } \Omega \times \Gamma},\\
	{(\varphi^i(T),\varphi^i_{\Gamma}(T)) =0}& {\text { in } \Omega \times \Gamma,}\\
	i=1,2.
	\end{array}\right.
	\end{equation}}Using the same ideas as in the linear case, we can prove the null controllablity of the optimality system \eqref{Eq25S0_}.  Then, by  some well-known fixed point arguments, we deduce the same results for the semilinear system \eqref{nl2.62_}.  More precisely, we have the next theorem.
\begin{theorem}Under the same assumptions of Theorem \ref{th4.1SN}, there exists a positive  weight function $\rho =\rho(t)$ blowing up at $t=T$  such that for every $y_{i,d}\in L^2(0,T; L^2(\omega_i))$ satisfying \eqref{Ass11SN}
	and  every $Y_0=(y_0,y_{\Gamma, 0} )\in\mathbb{L}^2$,  there  exist  controls $\hat{f}\in L^2(0,T;L^2(\omega))$ and  associated Nash quasi-equilibrium $(v^{\star}_1,v^{\star}_2)\in L^2(0,T;\break L^2(\omega_1))\times L^2(0,T;L^2(\omega_2))$
	such that the corresponding solution to \eqref{nl2.62_}   satisfies  $Y(T, \hat{f}, v^{\star}_1, v^{\star}_2)=0$.
\end{theorem}
 \section{Conclusion}
 In this paper we have studied a hierarchical control problem of heat equation with general dynamic boundary  conditions. Following the Stackelberg-Nash  strategy with one leader and two followers, we have proved, for each fixed leader, the existence and uniqueness of a Nash-equilibrium, and by means of an adjoint system, we have characterized the Nash-equilibrium, and we have  deduced an optimality system. By  suitable Carleman estimates, we have established an observability inequality which is  the key to deduce our  controllability result. The similar results are also obtained for  the  semilinear system.

\medskip
Received December 2020; revised June 2021.
\medskip

\end{document}